\documentclass[a4paper, 10pt]{amsart}

\usepackage{a4wide}
\usepackage{amsmath}
\usepackage{amssymb}
\usepackage{dsfont}

\newcommand{\e}{\varepsilon}

\newcommand{\R}{\mathbb{R}}
\newcommand{\N}{\mathbb{N}}
\renewcommand{\P}{\mathcal{P}}
\newcommand{\C}{\mathcal{C}}
\newcommand{\J}{\mathcal{J}}
\newcommand{\F}{\mathcal{F}}
\renewcommand{\H}{\mathcal{H}}
\newcommand{\E}{\mathcal{E}}

\newcommand{\K}{\mathcal{K}}
\newcommand{\G}{\mathcal{G}}

\usepackage{colortbl}

\newcommand{\pa}{\partial}
\newcommand{\lt}{\left}
\newcommand{\rt}{\right}
\newcommand{\bq}{\begin{equation}}
\newcommand{\eq}{\end{equation}}
\newcommand{\lal}{\langle}
\newcommand{\ral}{\rangle}
\newcommand{\mw}{W}

\renewcommand{\rho}{\varrho}

\newtheorem{proposition}{Proposition}
\newtheorem{theorem}{Theorem}
\newtheorem{corollary}{Corollary}

\theoremstyle{definition}
\newtheorem{definition}{Definition}
\newtheorem{remark}{Remark}

\title[Convergence for damped Euler equations with interaction forces]{Convergence to Equilibrium in Wasserstein distance for damped Euler equations with interaction forces}

\author[Carrillo]{Jos\'{e} A. Carrillo}
\address[Jos\'{e} A. Carrillo]{\newline Department of Mathematics, Imperial College London, 
    \newline London SW7 2AZ, United Kingdom}
\email{carrillo@imperial.ac.uk}

\author[Choi]{Young-Pil Choi}
\address[Young-Pil Choi]{\newline Department of Mathematics and Institute of Applied Mathematics, Inha University, 
    \newline Incheon 402-751, Republic of Korea}
\email{ypchoi@inha.ac.kr}

\author[Tse]{Oliver Tse}
\address[Oliver Tse]{\newline Department of Mathematics and Computer Science, Eindhoven University of Technology, 
\newline P.O. Box 513, 5600 MB Eindhoven, The Netherlands}
\email{o.t.c.tse@tue.nl}

\begin{document}


\begin{abstract}
We develop tools to construct Lyapunov functionals on the space of probability measures in order to investigate the convergence to global equilibrium of a damped Euler system under the influence of external and interaction potential forces with respect to the 2-Wasserstein distance. We also discuss the overdamped limit to a nonlocal equation used in the modelling of granular media with respect to the 2-Wasserstein distance, and provide rigorous proofs for particular examples in one spatial dimension.
\end{abstract}  

\maketitle


%
%
\section{Introduction}\label{sec:intro}

In this paper, we develop tools to analyse the large-time behavior of second-order dynamics that describe evolutions in the space of probability measures driven by a free energy $\F$. More precisely, we consider the evolution of probability measures $\mu$ with Lebesgue densities $\rho$ and their velocities $u$ described by {\em damped Euler systems} with damping parameter $\gamma>0$ of the form
\begin{align}\label{eq:euler}
\begin{aligned}
 \partial_t\rho_t + \nabla\cdot(\rho_t u_t) &= 0,\qquad (t,x) \in \R_+\times\R^d,\\
 \partial_t(\rho_t u_t) + \nabla\cdot(\rho_t u_t\otimes u_t) &= -\rho_t\nabla(\delta_{\mu}\F)(\mu_t) -\gamma \rho_t u_t,
\end{aligned}
\end{align}
subject to initial density and velocity conditions
\bq\label{ini_eq:euler}
(\rho_t, u_t)|_{t=0} = (\rho_0, u_0) \quad \mbox{for} \quad x \in \R^d,
\eq
where $\delta_\mu\F=\delta\F/\delta\mu$ is the variational derivative of a free energy $\F$ acting on probability measures $\mu$ that are absolutely continuous with respect to the Lebesgue measure on $\R^d$ with Radon--Nikodym derivative $d\mu/dx=\rho$, given by
\begin{align}\label{eq:entropy}
\F(\mu) := \int_{\R^d} U(\rho)\,dx + \int_{\R^d} V(x)\,d\mu(x) + \frac{1}{2}\iint_{\R^d\times\R^d} W(x-y)\,d\mu(x)\,d\mu(y).
\end{align}
Here, $U$ denotes an increasing function describing the internal energy of the density $d\mu/dx=\rho$, $V\colon\R^d\to\R$ and $\mw\colon\R^d\to\R$ are the confinement and the interaction potentials respectively. In this case, the variational derivative of $\F$ in \eqref{eq:entropy} is given by \cite{ambrosio2008gradient,villani2008optimal}
\[
 (\delta_{\mu}\F) (\mu) = U'(\rho) + V(x) + W\star\mu.
\]
In Section~\ref{sec:stationary} we will make sufficient assumptions on $U$, $V$ and $W$ in order to ensure the existence of a stationary measure $\mu_\infty$ for the damped Euler system \eqref{eq:euler} that may also be characterized as a minimizer of the free energy $\F$. In the sequel, we will identify the measure $\mu$ with its density $\rho$ as soon as the measure $\mu$ has a Lebesgue density.

While the well-posedness of \eqref{eq:euler} remains a challenging open question---even for restricted classes of initial data \eqref{ini_eq:euler}---and not dealt with in this paper, damped Euler systems without confining and interaction forces ($V=W\equiv 0$) have been investigated in multiple contexts. For instance, the global existence of BV and $L^\infty$ entropy weak solutions for the one-dimensional case were addressed in \cite{dafermos2009global,hsiao1998global} and \cite{ding1989convergence,huang2006asymptotic} respectively. The asymptotic behavior of solutions were also discussed in \cite{hsiao1992convergence,huang2005convergence,huang2006asymptotic,huang20111,nishihara2000lp}. For the multi-dimensional case, global existence and pointwise estimates of solutions based on the Green's function approach together with energy estimates were obtained in \cite{wang2001pointwise}, while the global existence of classical solutions and the large-time behavior of solutions were studied in \cite{bianchini2008dissipative,hanouzet2003global,sideris2003long} under the smallness assumptions on the initial data. We also refer to \cite{fang2009existence,liao2009lp} for the study of global well-posedness and asymptotic behavior of solutions based on the framework of Besov spaces. We refer the reader to \cite{chen2005euler} for a general survey of the Euler equations.

An initial attempt at proving equilibration results with explicit decay rates was conducted in \cite{klar2015entropy} for the case $U(s) = s\log s$, $V(x)=|x|^2/2$ and $W\equiv 0$. There, the authors used entropy dissipation methods to heuristically derive functional inequalities that provided the decay rates to equilibrium under relatively strong global regularity assumptions on $(\rho,u)$. The results in \cite{klar2015entropy} indicate a convergence behavior similar to spatially inhomogeneous entropy-dissipating kinetic equations where {\em hypocoercivity} of the operators involved played an important role in determining convergence to equilibrium \cite{desvillettes2005trend, DKMS13, dolbeault2015hypocoercivity,klar2014approximate,klar2017trend,villani2009hypocoercivity}. There, the exponential decay rate $\lambda=\lambda(\gamma)$ has the property that $\lambda\to 0$ as $\gamma\to 0$ and $\gamma\to+ \infty$, i.e., the best equilibration rate for \eqref{eq:euler} holds for some $\gamma\in(0,\infty)$.

A related equation is the well-known {\em aggregation-diffusion equation}
\begin{align}\label{eq:granular}
\partial_t\bar\rho_t = \nabla\cdot\bigl(\bar\rho_t\nabla(\delta_\mu\F)(\bar\rho_t)\bigr),\qquad (t,x) \in \R_+\times\R^d, 
\end{align}
with $\bar\rho_t$ a probability density on $\R^d$. The long-time asymptotics for \eqref{eq:granular} are given by the minimizer of the free energy $\F$ as $t\to\infty$, whenever the potentials are uniformly convex as in one of the earliest applications of these equations in granular media modelling \cite{BCP,BCCP,carrillo2003kinetic,LiTo}. Both equations,  \eqref{eq:euler} and \eqref{eq:granular}, also find numerous applications in mathematical biology and technology such as swarming of animal species, cell movement by chemotaxis, self-assembly of particles and dynamical density functional theory (DDFT)---see for instance \cite{review2,GPK,HP,reviewphysd} and the references therein.

Explicit equilibration rates for the aggregation-diffusion equation \eqref{eq:granular} have been derived using entropy dissipation methods \cite{carrillo2003kinetic, cordero2004inequalities} or, more recently, by using contraction estimates in the 2-Wasserstein distance \cite{bolley2014nonlinear, bolley2012convergence, bolley2013uniform} under convexity assumptions on the potentials. The entropy dissipation method is based on studying the time derivative of an appropriate Lyapunov functional along the flow generated by the aggregation-diffusion equation \eqref{eq:granular}, and using functional inequalities to bound the dissipation from below in terms of the Lyapunov functional. Actually, the solutions of \eqref{eq:granular} formally satisfy the free energy dissipation
\begin{equation}\label{eq:disgradflow}
\frac{d}{dt}\F(\bar\rho_t) = -\int_{\R^d} |\nabla(\delta_\mu\F)(\bar\rho_t)|^2 \bar\rho_t\,dx\,.
\end{equation}
Heuristically, one may view solutions of \eqref{eq:granular} as gradient flows of the free energy $\F$ on the space of probability measures, endowed with the 2-Wasserstein distance \cite{ambrosio2008gradient, carrillo2006contractions, otto2001geometry}. On the other hand, the method of using Wasserstein contraction estimates introduced in \cite{bolley2012convergence, bolley2013uniform,carrillo2006contractions}, is based on comparing the 2-Wasserstein distance with its dissipation along the evolution. This theory can deal with displacement convex functionals, as introduced in the seminal paper of McCann \cite{mccann1997convexity}, which include certain non uniformly convex potentials \cite{CDFLS,carrillo2012mass}. However, much less is known in terms of rates of convergence if uniform convexity of the potentials is not present, see \cite{BCLR2,BCL,BLR,CCP,CHVY} and the references therein for blow-up time, equilibrium solutions and qualitative covergence results.

In fact, equation \eqref{eq:granular} may be seen as an overdamped limit ($\gamma\to+\infty$) of the damped Euler equation \eqref{eq:euler} and have been studied in \cite{coulombel2007strong, junca2002strong} for the isothermal pressure law case ($U(s)=s\log s$, $V=W\equiv0$).  It was shown that the solutions to the damped isothermal Euler equations converge to that of the heat equation. For the isentropic pressure law case ($U(s)=s^m$, $m>1$, $V=W\equiv0$), the convergence to the porous media equation was discussed in \cite{lin2013strong} and \cite{marcati1990one} in one and multi-dimensions, respectively. Some particular cases in one spatial dimension have received a lot of attention due to the appearance of $\delta$-shocks, and their application to sticky particles \cite{brenier2013sticky,brenier1998sticky} or to consensus/contagion in swarming/crowd models \cite{BRSW,CCTT,carrillo2016pressureless}.

The objective of this paper is to develop contraction estimates in the 2-Wasserstein distance in the presence of uniform convexity of the potentials, in order to (a) prove convergence to equilibrium results for the damped Euler equations \eqref{eq:euler} in its full generality, and (b) to prove the overdamped limit ($\gamma\to\infty$) of \eqref{eq:euler} to \eqref{eq:granular} after suitable scaling. The general idea in handling both problems stems from viewing \eqref{eq:euler} as a damped harmonic oscillator for the pair $(\rho,u)$ with energy
\[
\H(\rho_t,u_t) := \F(\rho_t) + \frac{1}{2}\int_{\R^d} |u_t|^2\rho_t\,dx,
\]
which plays the role of a mathematical entropy and provides for a Lyapunov functional of \eqref{eq:euler}. Indeed, for smooth solutions $(\rho,u)$ of \eqref{eq:euler}, the identity
\begin{align}\label{eq:energy_estimate}
\frac{d}{dt}\H(\rho_t,u_t) = -\gamma\int_{\R^d} |u_t|^2 \rho_t\,dx \le 0,
\end{align}
holds. Although this inequality clearly states the dissipation of $\H$ with time $t\ge 0$, one cannot conclude the convergence to (global) equilibrium, since the right-hand side vanishes at local equilibria $(u_t \equiv 0)$. Integrating \eqref{eq:energy_estimate} gives the estimate
\[
\F(\rho_t) + \frac{1}{2}\int_{\R^d}|u_t|^2\rho_t\,dx + \gamma\int_0^t\int_{\R^d} |u_s|^2\rho_s\,dx\,ds \le \F(\rho_0) + \frac{1}{2}\int_{\R^d}|u_0|^2\rho_0\,dx,
\]
for all $t\ge 0$. A solution $(\rho, u)$ of \eqref{eq:euler} satisfying this estimate is called an {\em energy decaying solution} in the sequel. We will assume that these solutions exist globally in time with certain regularity for their velocity fields. We emphasize otherwise that our results hold without any smallness assumption on the initial data or closeness assumption to equilibrium solutions.

A good intuition for our strategy comes from the finite-dimensional setting. It is well-known that finite-dimensional gradient flows of uniformly convex energy landscapes enjoy exponential equilibration towards their unique global minimum. More precisely, assume $E:\R^d \longrightarrow \R$ to be a uniformly $\C^2$ convex function achieving its global minimum at zero with $D^2 E \geq \lambda I_d$, for some $\lambda>0$. Then, a good quantity to estimate the decay to zero of all solutions is given by the euclidean distance of a trajectory of the gradient flow $\dot x = -\nabla E(x)$ to the origin. Actually, one can show that
$$
\frac{d}{dt} |x(t)|^2 \leq -\lambda |x(t)|^2\qquad \mbox{for all } t\geq 0.
$$
The gradient flow $\dot x = -\nabla E(x)$ is the finite-dimensional counterpart of the aggregation-diffusion equation \eqref{eq:granular}. For the damped Euler system \eqref{eq:euler}, the finite dimensional counterpart is the classical damped oscillator $\dot x=v$, $\dot v=-\nabla E(x) -\gamma v$. Observe that the energy $E(x)$ is dissipated by the gradient flow $\dot x = -\nabla E(x)$, that is, $\frac{d}{dt} E(x) = -|\nabla E(x)|^2$ that resembles the gradient flow structure of \eqref{eq:granular} and its dissipation \eqref{eq:disgradflow}. In the case of the classical damped oscillator, we have the following dissipation of the total energy
$$
\frac{d}{dt}\left[\frac12 |v(t)|^2 + E(x(t))\right] = -\gamma |v(t)|^2 \qquad \mbox{for all } t\geq 0\,,
$$
that resembles \eqref{eq:energy_estimate}. Since the quantity $|x(t)|^2$ was a good measure of the equilibration of the gradient flow equation, it seems quite natural to check if it is also the case for the classical damped oscillator. In fact, one can show that
$$
\frac{d^2}{dt^2} |x(t)|^2 + \gamma \frac{d}{dt} |x(t)|^2 + \lambda |x(t)|^2 \leq 2|v(t)|^2
\qquad \mbox{for all } t\geq 0\,.
$$
This relation together with the energy identity implies the convergence, without rate, for the solutions of the classical damped oscillator towards the origin. Its proof will be discussed in Section 4 in the framework of solutions to the Euler equation \eqref{eq:euler}.

To analyse the evolution of probability measures, it is classical that the euclidean Wasserstein distance towards the global equilibrium of the free energy $\F$ plays the role of the euclidean distance in $\R^d$ to the origin. Therefore, motivated by the finite dimensional computation above and the work in \cite{klar2015entropy} (cf.~\cite{villani2009hypocoercivity}), we construct a Lyapunov functional based on the weighted sum of the energy $\H$, the 2-Wasserstein distance and its temporal derivative. In particular, we will require an estimate for the second-order temporal derivative of the 2-Wasserstein distance, which is provided by Theorem~\ref{thm:wasserstein_tt} in Section~\ref{sec:wasserstein}. Roughly speaking, it states that for solutions $(\rho,u)$ satisfying \eqref{eq:euler}, the second-order temporal derivative of the 2-Wasserstein distance between $\mu_t$, with density $\rho_t$, and any probability measure $\sigma$ with finite second moment, is given by
\[
 \frac{1}{2}\frac{d^+}{dt}\frac{d}{dt}W_2^2(\mu_t,\sigma) + \frac{\gamma}{2}\frac{d}{dt}W_2^2(\mu_t,\sigma) \le \int_{\R^d} |u_t|^2 d\mu_t - \int_{\R^d} \langle T_t(y)-y,\nabla(\delta_\mu\F)(\mu_t)\circ T_t(y)\rangle\,d\sigma,
\]
where $T_t\colon\R^d\to\R^d$ is an optimal transport map between $\mu_t$ and $\sigma$, satisfying $T_t\#\sigma=\mu_t$, i.e., $\mu_t$ is the push-forward of $\sigma$ under the map $T_t$.

When $V$ and $W$ satisfy certain $\lambda$-convexity assumptions ({\bf (H1)} and {\bf (H2)} below) and $\sigma=\mu_\infty$ is a sufficiently smooth minimizer of $\F$, i.e., $\mu_\infty$, with density $\rho_\infty$, satisfies $\mu_\infty\nabla(\delta_\mu\F)(\mu_\infty)=0$, then the previous estimate reduces to (cf. Corollary~\ref{cor:wasserstein_tt})
\[
 \frac{1}{2}\frac{d^+}{dt}\frac{d}{dt}W_2^2(\mu_t,\mu_\infty) + \frac{\gamma}{2}\frac{d}{dt}W_2^2(\mu_t,\mu_\infty) \le \int_{\R^d} |u_t|^2 d\mu_t - \frac{\lambda}{2}W_2^2(\mu_t,\mu_\infty),
\]
for some $\lambda>0$. It is this form of the estimate, along with estimate  \eqref{eq:energy_estimate}, that will be used to construct a strict Lyapunov function for the evolution, thereby resulting in the equilibration statements found in Theorems~\ref{thm:confinement} and \ref{thm:interaction} (see also Corollaries~\ref{cor:confinement} and \ref{cor:interaction}): For initial data $(\rho_0,u_0)$ with bounded energy and $W_2(\mu_0,\mu_\infty)<\infty$, one obtains
\[
 W_2(\mu_t,\mu_\infty)\longrightarrow 0\qquad\text{as\;\;$t\to\infty$}.
\]
In order to also deduce the convergence $\|u_t\|_{L^2(\mu_t)}\to 0$ as $t\to\infty$, one requires an additional assumption {\bf (H3)} on the relationship between $\F$ and $W_2$, and on the regularity of the solution.

A similar approach is used to prove the overdamped limit ($\gamma\to+\infty$) in Section~\ref{sec:relax}, where we compare a rescaled version of the solution $(\rho^\gamma,u^\gamma)$ to the Euler system \eqref{eq:euler} with the solution $\bar{\rho}$ of the granular media equation \eqref{eq:granular}. Since both $\rho^\gamma$ and $\bar{\rho}$ are time dependent, we extend the second-order estimate above to include measures $\sigma$ that evolve in time (Theorem~\ref{thm:wasserstein_tt_2}). This enables us to show in Theorem~\ref{thm:relaxation} that
\[
 \int_0^T W_2^2(\rho_t^\gamma dx,\bar\rho_t dx)\,dt\longrightarrow 0\qquad\text{as\;\;$\gamma\to\infty$},
\]
for any $T>0$ provided that the initial conditions are well-prepared.

The final section of the paper---Section~\ref{sec:1d}---gives rigorous proofs for particular examples in one dimension for which the calculus developed for equilibration in Section~\ref{sec:equilibration} provides (a) explicit exponential decay rates in the case when solutions are smooth, and (b) equilibration (without rates) whenever solutions form $\delta$-shocks in finite time. The example in (b) clearly illustrates the strength of the calculus even for the one dimensional case, where standard tools fail due to lack of regularity. In short, we show that all global in time Lagrangian solutions in the sense of \cite{brenier2013sticky} converge in $W_2$ towards a Dirac Delta at the center of mass of the initial density.

%
%

\section{Preliminary results}\label{sec:prelim}

We begin this section by introducing known results and stating restrictions on the free energy $\F$ that will be assumed throughout this paper. The next part of this section describes the general strategy applied to a toy example.

\begin{definition}
 Let $\P_2(\R^d)$ denote the set of Borel probability measures on $\R^d$ with finite second moment, i.e., $\int |x|^2d\mu<\infty$ for all $\mu\in\P_2(\R^d)$. The 2-Wasserstein distance between two measures $\mu$ and $\nu$ in $\P_2(\R^d)$ is defined as
 \[
 W_2(\mu,\nu) = \inf\nolimits_{\pi\in \Pi(\mu,\nu) } \left(\iint_{\R^d\times\R^d} |x-y|^2 d\pi(x,y)\right)^{1/2},
 \]
 where $\Pi(\mu,\nu)$ denotes the collection of all Borel probability measures on $\R^d\times\R^d$ with marginals $\mu$ and $\nu$ on the first and second factors respectively. The set $\Pi(\mu,\nu)$ is also known as the set of all couplings of $\mu$ and $\nu\in\P_2(\R^d)$. We further denote by $\Pi_0(\mu,\nu)$ the set of optimal couplings between $\mu$ and $\nu$. The Wasserstein distance defines a distance on $\P_2(\R^d)$ which metricizes the narrow convergence, up to a condition on the moments. We denote the set of probability measures having finite second moment with Lebesgue densities by $\P_2^{ac}(\R^d)$.
\end{definition}

\subsection{Existence of stationary measures}\label{sec:stationary}

To emphasize on the presentation of the equilibration method, we do not consider the most general assumptions to ensure the existence of minimizers of the free energy $\F$. Throughout this paper, we assume the conditions below:
\begin{enumerate}
	\item[{\bf (H1)}] $U \in \C([0,\infty)) \cap \C^2((0,\infty))$ with $U(0) = 0$, and the function $r\mapsto r^dU(r^{-d})$ is convex nonincreasing on $(0,\infty)$, or equivalently, 
	\[
	(d-1)p(r) \leq d r p^\prime(r)\quad \text{on } (0,\infty),\qquad p(r)=rU^\prime(r) - U(r),
	\]
	or $r \mapsto r^{-1+1/d}p(r)$ is nondecreasing on $(0,\infty)$.
	
	\item[{\bf (H2)}] $V$ and $W$ are $\C^1(\R^d)$ potentials on $\R^d$ with $W(-x)=W(x)$ for all $x\in\R^d$, satisfying
	\begin{align*}\left.
	\begin{aligned}
	\langle x-y,\nabla V(x)-\nabla V(y)\rangle &\ge c_V|x-y|^2 \\
	\langle x-y,\nabla W(x)-\nabla W(y)\rangle &\ge c_W|x-y|^2 
	\end{aligned}\quad\right\rbrace\quad \text{for all } x,y\in\R^d,
	\end{align*}
	with either $c_V>0$ and $c_V + c_W>0$ if $V\not\equiv 0$, or $c_V=0$ and $c_W>0$ if $V\equiv 0$.
\end{enumerate}

Under these conditions, we have the following result found in \cite{ambrosio2008gradient,bolley2013uniform,carrillo2003kinetic,mccann1997convexity}.

\begin{proposition}\label{prop:stationary}
	The free energy $\F\colon \P_2(\R^d)\to (-\infty, +\infty]$ defined by \eqref{eq:entropy} for absolutely continuous measures (w.r.t.~the Lebesgue measure) and by $+\infty$ otherwise achieves its minimum. A minimizer $\mu_\infty$ of $\F$ has a non-negative density $\rho_\infty$ on $\R^d$ satisfying
	\[
	 \nabla p(\rho_\infty) + \rho_\infty(\nabla V + \nabla W\star\rho_\infty)=0\qquad \text{a.e.}
	\]
In particular, we have
	\[
	 U^\prime(\rho_\infty) + V + W\star\rho_\infty = c\qquad \text{$\mu_\infty$-a.e.}
	\]
for some constant $c\in\R$.
\end{proposition}

\begin{remark}\label{rmk_1}
 For isothermal/isentropic flows in fluid dynamics the pressure law $p(\rho)$ is typically prescribed by
 \[
 p(\rho) = \rho^m \quad \mbox{with} \quad m \geq 1.
 \]
 In this case, the internal energy $U$ is uniquely given by
 \begin{align*}
 U(\rho) = \begin{cases}
 \rho\log \rho &\text{if } m=1,\\
 \rho^m/(m-1) &\text{if } m>1,
 \end{cases}
 \end{align*}
 and condition {\bf (H1)} takes the form $m\ge 1-1/d$ which is now classical. Consequently, the isentropic pressure satisfies condition {\bf (H1)} for any $m\ge 1$.
\end{remark}

\begin{remark}
	In the seminal work \cite{mccann1997convexity}, McCann showed that assumption {\bf (H1)} is equivalent to the requirement that $\int_{\R^d} U(\rho)\,dx$ is {\em (geodesically) displacement convex} on $(\P_2(\R^d), W_2)$. 
\end{remark}

\subsection{Equilibration of the center of mass}\label{sec:energy}
Here, we provide a simple construction of a strict Lyapunov functional for a toy example, for which we obtain exponential decay rates of the center of mass of $\mu$, with density $\rho$, where $(\rho,u)$ solves \eqref{eq:euler} with $V(x)=c_V|x-\bar x|^2/2$, $c_V>0$ for a given $\bar x\in\R^d$ and $W\equiv 0$. In this case, the free energy is given by
\[
 \F(\mu) = \int_{\R^d} U(\rho)\,dx + \frac{c_V}{2}\int_{\R^d} |x-\bar x|^2\,d\mu,
\]
with variational derivative $(\delta_{\mu}\F)(\mu) = U'(\rho) + c_V(x-\bar x)$.

By examining the evolution of the center of mass of $\mu_t$, we find that
\begin{align}\label{eq:oscillator}
 \frac{d}{dt} \int_{\R^d} x\,d\mu_t = \int_{\R^d} u_t \, d\mu_t,\qquad \frac{d}{dt} \int_{\R^d} u_t \, d\mu_t = -\int_{\R^d} \nabla V\,d\mu_t - \gamma \int_{\R^d} u_t\,d\mu_t.
\end{align}
which clearly resembles a damped harmonic oscillator. Note that the computations above are done component-wise, i.e.,
\[
 \frac{d}{dt} \int_{\R^d} x_i\,d\mu_t = \int_{\R^d} u_{i,t} \, d\mu_t\qquad\text{for all\; $i=1,\ldots,d$}.
\]
From this observation, one may easily deduce the exponential convergence of the center of mass of $\mu_t$ towards $\bar x\in\R^d$ at a rate $\lambda$ that has the properties $\lambda(\gamma)\to 0$ as $\gamma\to 0$ and $\gamma\to+\infty$. 

Indeed, it follows from \eqref{eq:oscillator} that 
\[
\frac{d^2}{dt^2} \int_{\R^d} (x-\bar x)\,d\mu_t  + \gamma \frac{d}{dt}\int_{\R^d} (x-\bar x)\,d\mu_t  + c_V\int_{\R^d} (x-\bar x)\,d\mu_t  = 0.
\]
Thus, from classical theory of differential equations, we may solve the second-order linear equation to obtain explicit decay rates. However, we use an alternative approach to estimate the decay rate, which simultaneously illustrates the basic idea behind our strategy. First of all, notice that
\[
 \frac12\frac{d}{dt}\lt[c_V\left|\int_{\R^d} (x-\bar x)\,d\mu_t\right|^2 + \left|\int_{\R^d} u_t \, d\mu_t\right|^2 \rt] = -\gamma \left|\int_{\R^d} u_t \, d\mu_t\right|^2,
\]
i.e., we are in the same conditions as in equation \eqref{eq:energy_estimate}. Now consider the temporal derivative of
\begin{align*}
 \J(t) &:=\alpha \left|\int_{\R^d} (x-\bar x)\,d\mu_t\right|^2 + 2\left(\int_{\R^d} (x-\bar x)\,d\mu_t\right)\cdot\left(\int_{\R^d} u_t \, d\mu_t\right) + \beta \left|\int_{\R^d} u_t \, d\mu_t\right|^2\\
 &=: \alpha \J_1(t) + 2\J_2(t) + \beta\J_3(t)
\end{align*}
where $\alpha,\beta>0$ are to be chosen appropriately. We note that for $\alpha\beta>1$, we have the equivalence 
\[
 p(\J_1 + \J_3) \le \J\le q(\J_1 + \J_3),
\]
for some constant $p,q>0$, depending only on $\alpha$ and $\beta$. Simple computations yield
\[
 \frac{d}{dt}\J(t) = -2c_V\J_1(t) + 2(\alpha -\beta c_V - \gamma) \J_2(t) - 2(\beta\gamma - 1)\J_3(t).
\]
Choosing $\beta=(1+c_V)/\gamma$ and $\alpha=\beta c_V + \gamma$, and using the fact that $\alpha\beta\ge 1+ c_V>1$, we obtain
\[
 \frac{d}{dt}\J(t) = -2 c_V(\J_1(t) +\J_3(t)) \le -2( c_V/q) \J(t).
\]
A simple application of the Grownwall inequality provides the exponential decay
\[
 \J(t) \le \J(0)e^{-(2c_V/q) t},\qquad q = \frac{(\alpha+\beta)+\sqrt{4+(\beta-\alpha)^2}}{2}.
\]
With the explicit choice of $\alpha$ and $\beta$, one easily examines the $\gamma$ dependent decay rate.

\begin{remark}
 Notice that while the above computations provide exponential convergence of the center of mass of $\mu$ and the momentum $u\mu$ towards $(\bar x,0)$, nothing can be said about $(\rho,u)$ itself.
\end{remark}

\section{Temporal derivatives of the Wasserstein distance}\label{sec:wasserstein}
Before we show any convergence results, we extend a basic result regarding the time derivatives of the Wasserstein distance between two evolving measures \cite{ambrosio2008gradient, bolley2012convergence, bolley2013uniform, villani2008optimal}.  
We begin by recalling a known result for the first temporal derivative \cite{ambrosio2008gradient}, \cite[Theorem~23.9]{villani2008optimal}.

\begin{proposition}\label{prop:wasserstein_t}
	Let $\mu,\nu\in\C([0,\infty),\P_2^{ac}(\R^d))$ be solutions of the continuity equations
	\[
	\partial_t\mu_t + \nabla\cdot(\mu_t\xi_t)=0,\qquad \partial_t\nu_t + \nabla\cdot(\nu_t\eta_t)=0,\qquad\text{in distribution},
	\]
	for locally Lipschitz vector fields $\xi$ and $\eta$ satisfying 
	\[
	\int_0^\infty \left(\int_{\R^d} |\xi_t|^2 d\mu_t + \int_{\R^d}|\eta_t|^2 d\nu_t\right)\,dt<\infty,
	\]
	then $\mu,\nu\in AC([0,\infty),\P_2^{ac}(\R^d))$ and for almost every $t\in(0,\infty)$,
	\begin{align}\label{eq:1st_derivative}
	\frac{1}{2}\frac{d}{dt}W_2^2(\mu_t,\nu_t) &= \iint_{\R^d\times\R^d} \langle x-y,\xi_t(x)-\eta_t(y)\rangle\,d\pi_t\\
	&=\int_{\R^d} \langle x-\nabla\varphi_t^*(x),\xi_t(x)\rangle\, d\mu_t + \int_{\R^d} \langle y-\nabla\varphi_t(y),\eta_t(y)\rangle\, d\nu_t \nonumber,
	\end{align}
	where $\pi_t\in\Pi_0(\mu_t,\nu_t)$ and $\nabla\varphi_t\#\nu_t=\mu_t$, $\nabla\varphi_t^*\#\mu_t=\nu_t$.
\end{proposition}

\begin{remark}
	Note that when $\xi$ and $\eta$ are globally Lipschitz vector fields, then the first temporal derivative \eqref{eq:1st_derivative} holds for all $t\in(0,\infty)$. We will implicitly use this fact in Theorem~\ref{thm:wasserstein_tt} below.
\end{remark}

\subsubsection*{Heuristical ideas}
For any two given measures $\mu_t,\nu_t\in\P_2^{ac}(\R^d)$, Brenier's theorem \cite{brenier1991polar,mccann1995existence} asserts the existence of a (proper) convex function $\varphi_t\colon \R^d\to(-\infty,+\infty]$ such that $\nabla\varphi_t\#\nu_t=\mu_t$ and $\nabla\varphi_t^*\#\mu_t=\nu_t$, where $\varphi_t(x)^*=\sup_{\R^d}\{\langle x,y\rangle -\varphi(y)\}$ is the Legendre--Fenchel dual of $\varphi_t$ satisying
\[
 (\nabla \varphi_t^* \circ \nabla \varphi_t )(y) = y\qquad \nu_t\text{-a.e.}
\]
In particular, we have the change of variables formula 
\[
 \int_{\R^d} g(t,\nabla\varphi_t(y))\,d\nu_t = \int_{\R^d} g(t,x)\,d\mu_t,
\]
for any test function $g\in \C_b(\R_+\times\R^d)$. Taking the temporal derivative gives
\begin{align}\label{eq:brenier_t_2}
 \int_{\R^d} \langle \nabla g(t,\nabla\varphi_t(y)),\partial_t\nabla\varphi_t(y)\rangle\,d\nu_t + \int_{\R^d} g(t,\nabla\varphi_t(y))\,d(\partial_t\nu_t) = \int_{\R^d} g(t,x)\,d(\partial_t\mu_t).
\end{align}
By choosing $g(t,x) = |x|^2/2 - \varphi_t^*(x)$, we obtain from \cite{villani2008optimal} (see also \cite{ambrosio2008gradient})
\begin{align*}
 \frac{1}{2}\frac{d}{dt}W_2^2(\mu_t,\nu_t) &= \frac{1}{2}\frac{d}{dt}\int_{\R^d} |\nabla\varphi_t(y)-y|^2d\nu_t \\
 &= \int_{\R^d} \langle \nabla\varphi_t(y)-y
,\partial_t\nabla\varphi_t(y)\rangle\,d\nu_t + \frac{1}{2}\int_{\R^d} |\nabla\varphi_t(y)-y|^2 d(\partial_t\nu_t)\\
 &= \int_{\R^d} g(t,x)\,d(\partial_t\mu_t) - \int_{\R^d} g(t,\nabla\varphi_t(y))\,d(\partial_t\nu_t) + \frac{1}{2}\int_{\R^d} |\nabla\varphi_t(y)-y|^2 d(\partial_t\nu_t).
\end{align*}
The last two terms on the right hand side may be expressed as
\[
 \int_{\R^d} \left(\frac{|y|^2}{2}- \langle \nabla\varphi_t(y),y\rangle + \varphi_t^*(\nabla\varphi_t(y)) \right) d(\partial_t\nu_t).
\]
Since $\varphi_t$ and $\varphi_t^*$ are duals of each other, we have that
\begin{align*}
 \varphi_t(y) + \varphi_t^*(\nabla\varphi_t(y)) = \langle \nabla\varphi_t(y),y\rangle.
\end{align*}
Consequently, we obtain
\[
 \frac{1}{2}\frac{d}{dt}W_2^2(\mu_t,\nu_t) = \int_{\R^d} g(t,x)\,d(\partial_t\mu_t) + \int_{\R^d} h(t,y)\,d(\partial_t\nu_t),
\]
with $h(t,y) = |y|^2/2 - \varphi_t(y)$. Finally, inserting the respective continuity equations and integrating by parts yield the required equality in \eqref{eq:1st_derivative}. The absolute continuity of $t\mapsto \mu_t$ and $t\mapsto \nu_t$ follows directly from \cite[Theorem~23.9]  {villani2008optimal}. In fact, they are shown to be H\"older-$1/2$ continuous.

The next part of this section is devoted to the representation of the second temporal derivative for the 2-Wasserstein distance along the flow of generic Euler equations of the form
\begin{align*}
 \partial_t\mu_t + \nabla\cdot(\mu_t\xi_t) &=0, \\
 \partial_t(\mu_t\xi_t) + \nabla\cdot(\mu_t\xi_t\otimes\xi_t) &= -\mu_t G_{\mu_t},
\end{align*}
where $G_\mu = G_\mu(t,x)$ is a sufficiently smooth function.

\begin{remark}
 In our particular case \eqref{eq:euler}, $G_\mu$ is related to the variational derivative of the free energy $\F$ and takes the form
 \[
  G_\mu(\mu,\xi) = \nabla (\delta_{\mu}\F)(\mu) + \gamma\xi = \nabla\big(U^\prime(\rho) + V + W\star\mu\big) + \gamma \xi\,,
 \]
with $\rho$ being the density of $\mu$.
\end{remark}

\subsubsection*{Heuristical ideas}
To simplify the notations, we set
\[
 T_t(y):=\nabla\varphi_t(y),\qquad T_t^*(x):=\nabla\varphi_t^*(x).
\]
Assuming that $\mu$ and $\nu$ satisfy the generic Euler equations
\begin{align*}
  \begin{aligned}
   \partial_t\mu_t + \nabla\cdot(\mu_t\xi_t) &=0, \\
   \partial_t(\mu_t\xi_t) + \nabla\cdot(\mu_t\xi_t\otimes\xi_t) &= -\mu_t G_{\mu_t}\,,
  \end{aligned}\qquad
  \begin{aligned}
   \partial_t\nu_t + \nabla\cdot(\nu_t\eta_t)&=0,\\
   \partial_t(\nu_t\eta_t) + \nabla\cdot(\nu_t\eta_t\otimes\eta_t) &= -\nu_t G_{\nu_t}\,,
  \end{aligned}
\end{align*}
with sufficiently smooth velocity fields $\xi_t$ and $\eta_t$, we deduce from \eqref{eq:brenier_t_2} that
\[
 \int_{\R^d} \langle \nabla g(t,T_t(y)),\partial_t T_t(y) + \nabla T_t(y)\eta_t(y) - \xi_t(T_t(y))\rangle\,d\nu_t = 0,
\]
for all smooth test functions $g\in\C^1_b(\R\times\R^d)$. This essentially means
\begin{align}\label{eq:characteristics}
 \partial_t T_t(y) + \nabla T_t(y)\eta_t(y) = \xi_t(T_t(y))\qquad \nu_t\text{-a.e.}
\end{align}
Let us now consider the second temporal derivative of the Wasserstein distance, i.e., we formally take the temporal derivative of \eqref{eq:1st_derivative} to obtain
\begin{align}\label{eq:wasserstein_tt_2}
\begin{aligned}
 \frac{1}{2}\frac{d^2}{dt^2}W_2^2(\mu_t,\nu_t) = &\, -\int_{\R^d} \langle \partial_t T_t^*(x),\xi_t(x)\rangle\, d\mu_t  - \int_{\R^d} \langle \partial_t T_t(y),\eta_t(y)\rangle\, d\nu_t \\
 &+ \int_{\R^d} \langle x-T_t^*(x),\partial_t(\mu_t\xi_t)\rangle  + \int_{\R^d} \langle y-T_t(y),\partial_t(\nu_t\eta_t)\rangle.
\end{aligned}
\end{align}
For the first term, we notice the fact that
\[
 0=\partial_t(T_t^*\circ T_t)(y) = \partial_t T_t^*(T_t(y)) + \nabla T_t^*(T_t(y))\partial_t T_t(y).
\]
Using the previous equality, \eqref{eq:characteristics} and the fact that $\nabla T_t^*(T_t(y))\nabla T_t(y) = \mathbb{I}_d$, we obtain
\begin{align*}
 \int_{\R^d} \langle \partial_t T_t^*(x),\xi_t(x)\rangle\, d\mu_t &= \int_{\R^d} \langle \partial_t T_t^*(T_t(y)),\xi_t(T_t(y))\rangle\, d\nu_t \\
 &= -\int_{\R^d} \langle \nabla T_t^*(T_t(y))\partial_t T_t(y),\xi_t(T_t(y))\rangle\, d\nu_t \\
 &= \int_{\R^d} \langle \nabla T_t^*(T_t(y))\Big( \nabla T_t(y)\eta_t(y) - \xi_t(T_t(y))\Big),\xi_t(T_t(y))\rangle\, d\nu_t \\
 &= \int_{\R^d} \langle \eta_t(y),\xi_t(T_t(y))\rangle\, d\nu_t - \int_{\R^d} \langle\nabla T_t^*(x)\xi_t(x),\xi_t(x)\rangle\, d\mu_t.
\end{align*}
Computing the second term analogously gives
\[
 \int_{\R^d} \langle \partial_t T_t(y),\eta_t(y)\rangle\, d\nu_t = -\int_{\R^d} \langle \nabla T_t(y)\eta_t(y),\eta_t(y)\rangle\, d\nu_t + \int_{\R^d} \langle\xi_t(T_t(y)),\eta_t(y)\rangle\, d\nu_t \,.
\]
The last two terms of \eqref{eq:wasserstein_tt_2} may be handled simultaneously to give
\begin{align*}
 & \int_{\R^d} \langle x-T_t^*(x),\partial_t(\mu_t\xi_t)\rangle  + \int_{\R^d} \langle y-T_t(y),\partial_t(\nu_t\eta_t)\rangle \\
 &\hspace*{6em}= \int_{\R^d} \langle (\mathbb{I}_d - \nabla T_t^*(x))\xi_t(x),\xi_t(x)\rangle\,d\mu_t - \int_{\R^d} \langle x-T_t^*(x),G_{\mu_t}\rangle\,d\mu_t \\
 &\hspace*{7em} + \int_{\R^d} \langle (\mathbb{I}_d - \nabla T_t(y))\eta_t(y),\eta_t(y)\rangle\,d\nu_t- \int_{\R^d} \langle y-T_t(y),G_{\nu_t}\rangle\,d\nu_t \,.
\end{align*}
The previous heuristic ideas can now be turned into the following theorem under the right assumptions.

\begin{theorem}\label{thm:wasserstein_tt}
	Let $\mu\in\C([0,T),\P_2^{ac}(\R^d))$ satisfy the Euler type equation
	\begin{align*}\left.
	\begin{aligned}
	\partial_t\mu_t + \nabla\cdot(\mu_t\xi_t) &=0, \\
	\mu_t\big(\partial_t\xi_t + \xi_t\cdot\nabla\xi_t\big) &= -\mu_t G_{\mu_t},
	\end{aligned}\;\;\right\rbrace\;\; \text{in distribution},
	\end{align*}
	with locally in $t>0$ and globally in $x\in\R^d$ Lipschitz vector field $x\mapsto \xi_t(x)$ satisfying
	\[
	t\mapsto\|\xi_t\|_{L^2(\mu_t)},\|G_{\mu_t}(t,\cdot)\|_{L^2(\mu_t)}\in\C([0,\infty))\cap L^2([0,\infty)).
	\]
	For any $\sigma\in\P_2(\R^d)$, define
	\[
	\mathcal{K}(\mu_t,\sigma) := \frac{d}{dt}W_2^2(\mu_t,\sigma) = 2\iint_{\R^d\times\R^d}\langle x-y,\xi_t(x)\rangle\,d\pi_t.
	\]
	Then for any $T>0$ the following inequality holds:
	\begin{align*}
	\mathcal{K}(\mu_T,\sigma) \le \mathcal{K}(\mu_0,\sigma) + 2\int_0^T \left(\int_{\R^d} |\xi_t|^2 \,d\mu_t - \iint_{\R^d\times\R^d} \langle x-y,G_{\mu_t}(t,x)\rangle\,d\pi_t\right) dt,
	\end{align*}
	for the optimal transference plan $\pi_t\in \Pi_0(\mu_t,\sigma)$. In particular, we obtain
	\[
	\frac{1}{2}\frac{d^+}{dt}\frac{d}{dt}W_2^2(\mu_t,\sigma) = \frac{1}{2}\frac{d^+}{dt}\mathcal{K}(\mu_t,\sigma) \le \int_{\R^d} |\xi_t|^2 \,d\mu_t - \iint_{\R^d\times\R^d} \langle x-y,G_{\mu_t}(t,x)\rangle\,d\pi_t,
	\]
	where $d^+/dt$ denotes the upper derivative in almost every $t>0$.
\end{theorem}
\begin{proof}
	{\em Step 1:} For some fixed $t\in(0,\infty)$ let
	\begin{align}\label{eq:flow}
	\partial_\tau \Phi_\tau(x) = (\xi_{t+\tau}\circ \Phi_\tau)(x), \quad \Phi_0 = x\qquad \text{for\, $\mu_t$-a.e.~$x$},
	\end{align}
	be the well-defined global in $\tau\in(-t,\infty)$ corresponding Lipschitz flow and set
	\begin{align*}
	\mu_{t + h} = \Phi_{h}\#\mu_{t}, \qquad \mu_{t-h} = \Phi_{-h}\#\mu_{t},
	\end{align*}
	for each $h\in(0,t)$. Furthermore, taking the temporal derivative of \eqref{eq:flow} and using the momentum equation for $\mu_t\xi_t$ provides the representation
	\begin{align}\label{eq:flow_d}
	\partial_\tau^2\Phi_\tau(x) = \big(\partial_t\xi_{t+\tau} + \xi_{t+\tau}\cdot\nabla_x \xi_{t+\tau}\big)\circ \Phi_\tau(x) = -G_{\mu_t}(t+\tau,\Phi_\tau(x))\qquad \text{for\, $\mu_t$-a.e.~$x$},
	\end{align}
	and a.e.~$\tau\in(-t,\infty)$, which will be used in the following steps.
	
	Finally let $\pi_t\in\Pi_0(\mu_t,\sigma)$ be the unique optimal plan between $\mu_t$ and $\sigma\in \P_2(\R^d)$. One clearly sees that $\pi_{t}^\tau := (\Phi_{\tau}\times id)\#\pi_t$ induces a transference plan between $\mu_{t + \tau}$ and $\sigma\in\P_2(\R^d)$.
	
	{\em Step 2:} For some fixed $t\in(0,\infty)$ and $h\in(0, t)$, consider the finite difference
	\begin{align*}
	\Delta_h \mathcal{K}(\mu_t,\sigma) := (D_{\frac{h}{2}}^\text{sym}D_{\frac{h}{2}}^\text{sym}W_2^2)(\mu_t,\sigma),
	\end{align*}
	where $D_\tau^\text{sym}$ denotes the symmetric difference operator with step $\tau>0$, i.e.,
	\[
	(D_\tau^\text{sym}W_2^2)(\mu_t,\sigma) := \frac{1}{2\tau}\lt(W_2^2(\mu_{t + \tau},\sigma) - W_2^2(\mu_{t - \tau},\sigma) \rt).
	\]
	Thus, we explicitly obtain
	\begin{align*}
	\Delta_h \mathcal{K}(\mu_t,\sigma) &= \frac{1}{h^2}\Big( W_2^2(\mu_{t+h},\sigma) - 2 W_2^2(\mu_{t},\sigma) + W_2^2(\mu_{t-h},\sigma)\Big),
	\end{align*}
	i.e., $\Delta_h$ is a second order symmetric difference operator. Notice that, by passing to the limit $h\to 0$ in $\Delta_h\mathcal{K}(\mu_t,\sigma)$, one obtains
	\[
	\lim_{h\to 0}\Delta_h \mathcal{K}(\mu_t,\sigma) = \frac{d^2}{dt^2}W_2^2(\mu_t,\sigma) = \frac{d}{dt}\K(\mu_t,\sigma),
	\]
	whenever the right-hand side is well-defined. Unfortunately, the existence of a second temporal 2-Wasserstein derivative may not be easily justified. Instead, we proceed with the finite difference computations, while mimicking the formal differential computations.
	
	Recall that $\pi_{t}^\tau := (\Phi_{\tau}\times id)\#\pi_t\in \Pi(\mu_{t+\tau},\sigma)$ for any $\tau\in[-h,h]$. Hence, 
	\[
	W_2^2(\mu_{t+\tau},\sigma) \le \iint_{\R^d\times\R^d} |x-y|^2 d\pi_t^\tau = \iint_{\R^d\times\R^d} |\Phi_\tau(x)-y|^2 d\pi_t.
	\]
	Consequently, for any $h\in(0,t)$, we have
	\begin{align*}
	\Delta_h \mathcal{K}(\mu_t,\sigma) &= \frac{1}{h^2} \Big( W_2^2(\mu_{t+h},\sigma) - 2 W_2^2(\mu_{t},\sigma) + W_2^2(\mu_{t-h},\sigma) \Big) \\
	&\le \frac{1}{h^2}\iint_{\R^{d}\times\R^d} |\Phi_h(x)-y|^2 - 2|x-y|^2 + |\Phi_{-h}(x)-y|^2\, d\pi_t = ({\rm I}).
	\end{align*}
	Using the fundamental theorem of calculus and Jensen's inequality, (I) may be reformulated as
	\begin{align*}
	({\rm I}) = &\,\frac{1}{h^2}\iint_{\R^{d}\times\R^d} |\Phi_h(x)-x|^2 + 2\langle \Phi_h(x)-2x+\Phi_{-h}(x),x-y\rangle + |\Phi_{-h}(x)-x|^2\, d\pi_t \\
	\le &\,\frac{1}{h}\int_{-h}^h \int_{\R^d}|\partial_\tau\Phi_\tau(x)|^2 d\mu_t\,d\tau + \frac{1}{h^2}\int_0^h  \iint_{\R^d\times\R^d} (h-\tau)\langle x-y,\pa_\tau^2\Phi_\tau(x)\rangle\,d\pi_t\,d\tau\\
	& + \frac{1}{h^2}\int_{-h}^0 \iint_{\R^d\times\R^d} (h+\tau)\langle x-y, \partial_\tau^2\Phi_\tau(x)\rangle\,d\pi_t\,d\tau\\
	=&\, \int_{-1}^1 \int_{\R^d}|\xi_{t+sh}(x)|^2 d\mu_{t+sh}\,ds - \int_{0}^1 (1-s)\iint_{\R^d\times\R^d} \langle x-y, G_{\mu_{t+sh}}(t+sh,\Phi_{sh}(x))\rangle\,d\pi_t\,ds\\
	& - \int_{-1}^0 (1+s)\iint_{\R^d\times\R^d} \langle x-y, G_{\mu_{t+sh}}(t+sh,\Phi_{sh}(x))\rangle\,d\pi_t\,ds,
	\end{align*}
	where we inserted the representation \eqref{eq:flow_d} in the last equality.
	
	{\em Step 3:} For a fixed $T>0$, we choose $N\in\N$ in such a way that $h=T/N$. Now consider a family $\{\mu_{nh}\}_{n\in I_N}\subset\P_2^{ac}(\R^d)$ for $I_N=\{0,\ldots,N\}\subset \N_0$, defined recursively by $\mu_{(n+1)h} = \Phi_h^n\#\mu_{nh}$ for $n\in I_N$, where $\Phi_h^n$ satisfies
	\[
	\partial_\tau \Phi_\tau^n(x) = (\xi_{nh+\tau}\circ \Phi_\tau^n)(x), \quad \Phi_0^n = x\qquad \text{for\, $\mu_{nh}$-a.e.~$x$}.
	\]
	and $\tau\in(-h,h)$. Then, for each $n\in I_N$, {\em Step 2} provides the inequality
	\begin{align}\label{eq:K_inequality}
	\Delta_h \mathcal{K}(\mu_{nh},\sigma) \le &\,\int_{-1}^1 \int_{\R^d}|\xi_{(n+s)h}(x)|^2 d\mu_{(n+s)h}\,ds \nonumber\\
	&- \int_{0}^1 (1-s)\iint_{\R^d\times\R^d} \langle x-y, G_{\mu_{(n+s)h}}((n+s)h,\Phi_{sh}^n(x))\rangle\,d\pi_{nh}\,ds \nonumber\\
	&- \int_{-1}^0(1+s) \iint_{\R^d\times\R^d} \langle x-y, G_{\mu_{(n+s)h}}((n+s)h,\Phi_{sh}^n(x))\rangle\,d\pi_{nh}\,ds \nonumber\\
	&=: ({\rm A}) + ({\rm B}) + ({\rm C}).
	\end{align}
	Multiplying the inequality with $h$ and summing over $n\in I_N$ yields for the left-hand side
	\[
	\sum_{n=1}^{N-1} h\,\Delta_h \mathcal{K}(\mu_{nh},\sigma) = (D_{\frac{h}{2}}^\text{sym}W_2^2)(\mu_{(N-\frac{1}{2})h},\sigma) - (D_{\frac{h}{2}}^\text{sym}W_2^2)(\mu_{\frac{h}{2}},\sigma).
	\]
	Before proceeding, we first note that the fundamental theorem of calculus and the representation of the first temporal 2-Wasserstein derivative provided in Proposition~\ref{prop:wasserstein_t} yields
	\begin{align*}
	(D_{\frac{h}{2}}^\text{sym}W_2^2)(\mu_{(n+\frac{1}{2})h},\sigma) &= \frac{1}{h}\Big(W_2^2(\mu_{(n+1)h},\sigma) - W_2^2(\mu_{nh},\sigma) \Big) =\frac{1}{h}\int_{nh}^{(n+1)h} \frac{d}{d\tau}W_2^2(\mu_{\tau},\sigma)\,d\tau \\
	&= \frac{2}{h}\int_{nh}^{(n+1)h} \iint_{\R^d\times\R^d}\langle x-y,\xi_{\tau}(x)\rangle\,d\pi_{\tau}\,d\tau \\
	&= 2\int_{0}^{1} \iint_{\R^d\times\R^d}\langle x-y,\xi_{(n+s)h}(x)\rangle\,d\pi_{(n+s)h}\,ds\qquad\text{for any\, $n\in I_N$}.
	\end{align*}
	Therefore, passing to the limit $N\to\infty$ with $T = hN$ gives
	\begin{align*}
	\lim_{N\to \infty}\sum_{n=1}^{N-1} h\,\Delta_h \mathcal{K}(\mu_{nh},\sigma) &= 2\left( \iint_{\R^d\times\R^d}\langle x-y,\xi_{T}(x)\rangle\,d\pi_{T} - \iint_{\R^d\times\R^d}\langle x-y,\xi_{0}(x)\rangle\,d\pi_{0}\right) \\
	&= \mathcal{K}(\mu_{T},\sigma) - \mathcal{K}(\mu_{0},\sigma),
	\end{align*}
	which holds due to Lebesgue's dominated convergence theorem. On the other hand, the following convergences hold for the terms on the right-hand side of \eqref{eq:K_inequality}:
	\begin{gather*}
	\sum_{n=1}^{N-1} h({\rm A}) \longrightarrow 2\int_0^T \int_{\R^d}|\xi_{t}(x)|^2 d\mu_{t}\,dt,\quad \sum_{n=1}^{N-1} h({\rm B}) \longrightarrow \frac{1}{2}\int_0^T \iint_{\R^d\times\R^d} \langle x-y, G_{\mu_t}(t,x)\rangle\,d\pi_{t}\,dt\\
	\sum_{n=1}^{N-1} h({\rm C}) \longrightarrow \frac{3}{2}\int_0^T \iint_{\R^d\times\R^d} \langle x-y, G_{\mu_t}(t,x)\rangle\,d\pi_{t}\,dt.
	\end{gather*}
	These convergences hold simply by definition of Riemann integrable functions and the assumed regularity of $\xi$ and $G_\mu$. 
	Indeed, due to the assumed continuity of $f(t):= \|\xi_t\|_{L^2(\mu_t)}^2$, we know that $f$ is Riemann integrable on $[0,T]$. Therefore, the corresponding upper Darboux sum satisfies
		\[
		 \sum_{n=0}^{N-1} h \sup_{s\in[0,1]}f((n+s)h)\longrightarrow \int_0^T f(t)\,dt\qquad\text{as\; $N\to\infty$}.
		\]
		In particular, we have for any $s\in[0,1]$, that
		\[
		 I_f^{N}(s):=\sum_{n=0}^{N-1} hf((n+s)h) \longrightarrow \int_0^T f(t)\,dt\qquad\text{as\; $N\to\infty$}.
		\]
		Furthermore, we may reformulate the sum to obtain
	\begin{align*}
		\sum_{n=1}^{N-1} h({\rm A}) &= \sum_{n=1}^{N-1} h\int_{-1}^1 \int_{\R^d}|\xi_{(n+s)h}(x)|^2 d\mu_{(n+s)h}\,ds = \int_{-1}^1 \sum_{n=1}^{N-1} hf((n+s)h)\,ds \\
		&= \int_{0}^1 \sum_{n=1}^{N-1} hf((n+s)h)\,ds + \int_{-1}^0 \sum_{n=1}^{N-1} hf((n+s)h)\,ds \\
		&= \int_{0}^1 \sum_{n=1}^{N-1} hf((n+s)h)\,ds + \int_0^1 \sum_{n=0}^{N-2} hf((n+s)h)\,ds \\
		&= 2\int_{0}^1 I_f^{N}(s)\,ds - \int_0^1 hf(sh)\,ds - \int_0^1 h f(T-(1-s)h)\,ds.
	\end{align*}
	It is not hard to see that $|I_f^{N}(s)|\le c$ with some constant $c>0$ for all $N\gg 1$ sufficiently large, and that the two last terms vanish as $N\to\infty$ due to the boundedness of $f$ on $[0,T]$. Therefore, an application of the Lebesgue dominated convergence yields
	\[
	 \lim_{N\to \infty} \sum_{n=1}^{N-1} h({\rm A}) = 2\int_{0}^1 \lim_{N\to\infty}I_f^{N}(s)\,ds = 2\int_{0}^1 \int_0^T f(t)\,dt\,ds = 2\int_0^T f(t)\,dt,
	\]
	as required. The convergence of the sums for $({\rm B})$ and $({\rm C})$ may be shown in a similar fashion.
	
	Collecting all the terms together finally yields the statement. 
\end{proof}

Mimicking the strategy of the proof to Theorem~\ref{thm:wasserstein_tt}, we arrive at the following result.

\begin{theorem}\label{thm:wasserstein_tt_2}
	Let $\mu_t$ and $\nu_t\in\P_2^{ac}(\R^d)$, $t\ge 0$, satisfy Euler type equations of the form
	\begin{align*}\left.
	\begin{aligned}
	\partial_t\mu_t + \nabla\cdot(\mu_t\xi_t) &=0, \\
	\mu_t\big(\partial_t\xi_t + \xi_t\cdot\nabla\xi_t\big) &= -\mu_t G_\mu,
	\end{aligned}\qquad
	\begin{aligned}
	\partial_t\nu_t + \nabla\cdot(\nu_t\eta_t)&=0,\\
	\nu_t\big(\partial_t\eta_t + \eta_t\cdot\nabla\eta_t\big) &= -\nu_t G_\nu,
	\end{aligned}\;\;\right\rbrace\;\; \text{in distribution},
	\end{align*}
	with locally in $t>0$ and globally in $x\in\R^d$ Lipschitz vector fields $x\mapsto\xi_t(x),\,\eta_t(x)$ satisfying
	\begin{align*}\label{eq:Gintegrability}
	t\mapsto\|\xi_t\|_{L^2(\mu_t)},\|\eta_t\|_{L^2(\nu_t)},\|G_{\mu_t}(t,\cdot)\|_{L^2(\mu_t)},\|G_{\nu_t}(t,\cdot)\|_{L^2(\nu_t)}\in\C([0,\infty))\cap L^2([0,\infty)).
	\end{align*}
	Then for any $T>0$ the following inequality holds:
	\begin{align*}
	\mathcal{K}(\mu_T,\nu_T) \le \mathcal{K}(\mu_0,\nu_0) + 2\int_0^T \int_{\R^d\times\R^d} |\xi_t(x)-\eta_t(y)|^2 - \langle x-y,G_{\mu_t}(t,x)-G_{\nu_t}(t,y)\rangle\,d\pi_t\, dt,
	\end{align*}
	for the optimal transference plan $\pi_t\in \Gamma_0(\mu_t,\nu_t)$. In particular, we obtain
	\[
	\frac{1}{2}\frac{d^+}{dt}\mathcal{K}(\mu_t,\nu_t) \le \iint_{\R^d} |\xi_t(x)-\eta_t(y)|^2 d\pi_t - \iint_{\R^d\times\R^d} \langle x-y,G_{\mu_t}(t,x)-G_{\nu_t}(t,y)\rangle\,d\pi_t\,.
	\]
\end{theorem}

A direct consequence of Theorem~\ref{thm:wasserstein_tt_2} is the following result.

\begin{corollary}\label{cor:wasserstein_tt}
Let $(\rho,u)$ be an energy decaying solution of the Euler equations \eqref{eq:euler} with $\mu$, with density $\rho$, and $u$ satisfying additionally the assumptions of Theorem~\ref{thm:wasserstein_tt}. Furthermore, let $\nu$, with density $\omega$, satisfy $\nabla p(\omega) + \omega(\nabla V + \nabla \mw \star \nu) = 0$ $dx$-almost everywhere. Suppose that $p'(\omega)\in L^2(\nu)$ and $p'(\rho_t) \in L^2(\mu_t)$ for almost every $t\in(0,\infty)$. Then,  the following inequality 
\[
 \frac{1}{2}\frac{d^+}{dt}\frac{d}{dt}W_2^2(\mu_t,\nu) \le \|\xi_t\|^2_{L^2(\mu_t)} - \frac{\gamma}{2}\frac{d}{dt}W_2^2(\mu_t,\nu) - J_V(\mu_t|\nu) - J_\mw(\mu_t|\nu),
\]
holds for almost every $t>0$. Here, the functionals $J_V$ and $J_W$ are defined by
\begin{align*}
J_V(\mu_t|\nu)& := \int_{\R^d} \lal y-T_t(y), \nabla V(y)-\nabla V(T_t(y))\ral \,d\nu(y),\cr
J_\mw(\mu_t|\nu) &:= \frac12\iint_{\R^d\times\R^d} \lal (y-\hat y) - T_t(y) - T_t(\hat y), \cr
& \hspace*{8em} \nabla \mw(y-\hat y) - \nabla \mw\lt(T_t(y) - T_t(\hat y)\rt) \ral\, d\nu(y) d\nu(\hat y),
\end{align*}
where $T_t$ is the unique transport map satisfying $T_t \# \nu = \mu_t$.
\end{corollary}

\begin{proof}
 We make use of Theorem~\ref{thm:wasserstein_tt} with $(\nu_t,\eta_t)=(\nu,0)$ and
 \[
  G_\mu=\nabla \big( U^\prime(\rho) + V + W\star\mu\big) + \gamma \xi,\qquad G_\nu = \nabla \big( U^\prime(\omega) + V + W\star\nu\big).
 \]
 In fact, we have that $\omega\, G_\nu\equiv 0$ a.e. Indeed, by construction we find
 \[
  \omega\,G_\nu = \omega\nabla \big( U^\prime(\omega) + V + W\star\nu\big) = \nabla p(\omega) + \omega\nabla \big(V + W\star\nu\big) = 0\qquad \text{a.e.}
 \]
 We begin by computing the term
 \begin{align*}
   \int_{\R^d} \langle y-T_t(y),G_{\nu_t}(t,y)\rangle\,d\nu &= \int_{\R^d} \langle y-T_t(y),\nabla p(\omega)\rangle\,dy +\int_{\R^d} \langle y-T_t(y),\nabla \big(V + W\star\nu\big)\rangle\,d\nu,
  \end{align*} 
 where we used the fact that $\omega\nabla U^\prime(\omega)=\nabla p(\omega)$. Similarly, we obtain
 \begin{align*}
   \int_{\R^d} \langle y-T_t(y),G_{\mu_t}(t,T_t(y))\rangle\,d\nu &= \int_{\R^d} \langle T_t^*(x)-x,\nabla p(\rho_t)\rangle\,dx \\
   &\hspace*{-6em}+ \gamma \int_{\R^d} \langle T_t^*(x)-x,\xi_t(x)\rangle\,d\mu_t +\int_{\R^d} \langle T_t^*(x)-x,\nabla \big(V + W\star\mu_t\big)\rangle\,d\mu_t.
 \end{align*}
  Subtracting the equation above from the previous one, we get
  \begin{align*}
   &\int_{\R^d} \langle y-T_t(y),G_{\nu_t}(t,y)-G_{\mu_t}(t,T_t(y))\rangle\,d\nu \\
   &\hspace*{5em}= \int_{\R^d} \langle y-T_t(y),\nabla p(\omega)\rangle\,dy - \int_{\R^d} \langle T_t^*(x)-x,\nabla p(\rho_t)\rangle\,dx + \frac{\gamma}{2}\frac{d}{dt}W_2^2(\mu_t,\nu) \\
   &\hspace*{8em}+ J_V(\mu_t|\nu) + \int_{\R^d} \langle y-T_t(y),(\nabla W\star\nu)(y)-(\nabla W\star\mu_t)(T_t(y))\rangle\,d\nu\\
   &\hspace*{5em}= I_1 + \frac{\gamma}{2}\frac{d}{dt}W_2^2(\mu_t,\nu) + J_V(\mu_t|\nu) + I_2.
  \end{align*}
 In order to deal with $I_1$, we proceed by a ``weak" integration by parts as in \cite{lisini2009nonlinear,bolley2012convergence,carrillo2006contractions}. We give some details on this. We consider a smooth cut-off function $\chi_R\in\C_0^\infty(\R^d)$ satisfying the properties
  \[
   0\le \chi_R\le 1,\quad |\nabla\chi_R|\le \frac{C}{R},\quad \chi_R\equiv 1\;\;\text{on\, $B_R(0)$},\quad \chi_R\equiv 0\;\;\text{on\, $\R^d\setminus B_{2R}(0)$}.
  \]
Under the assumption $p'(\rho_t)\in L^2(\mu_t)$, we obtain from {\bf (H1)}
  \begin{align*}
  \int_{\R^d} p(\rho_t)\,dx &\le \frac{d}{d-1}\int_{\R^d} \rho_t\, p'(\rho_t)\,dx \le \frac{d}{d-1} \left(\int_{\R^d}|p'(\rho_t)|^2 d\mu_t\right)^{1/2}, \\
  \int_{\R^d} |T_t^*-id|\,p(\rho_t)\,dx &\le \frac{d}{d-1}\int_{\R^d} |T_t^*-id|\,p'(\rho_t)\,d\mu_t \le \frac{d}{d-1} W_2(\mu_t,\nu)\left(\int_{\R^d} |p'(\rho_t)|^2d\mu_t\right)^{1/2},
  \end{align*}
and therefore, $p(\rho_t)$ and $ |T_t^*-id|\,p(\rho_t)$ are in $L^1(\R^d)$ a.e. $t\in(0,\infty)$. On the other hand, the assumption on $G_\mu(t,\cdot)$ provides
  \[
  \int_{\R^d} |\nabla p(\rho_t)|\,dx = \int_{\R^d} |\nabla U'(\rho_t)|\,d\mu_t\le \|\nabla U'(\rho_t)\|_{L^2(\mu_t)}\qquad\text{a.e. $t\in(0,\infty)$},
  \]
and therefore, $p(\rho_t)\in W^{1,1}(\R^d)$ a.e. $t\in(0,\infty)$.

Since $\chi_R(T_t^*-id)\in L^\infty(\R^d)\cap BV(\R^d)$, we perform integration by parts to obtain
  \begin{align}\label{eq:strong_int}
  	\begin{aligned}
  		-\int_{\R^d} \chi_R\langle T_t^*-id,\nabla p(\rho_t)\rangle\,dx &\ge \int_{\R^d} \tilde\nabla\cdot\big(\chi_R\,(T_t^*-id)\big)\,p(\rho_t)\,dx \\
  		&= \int_{\R^d} \chi_R\,(\tilde\nabla \cdot T_t^*)\,p(\rho_t)\,dx - d\int_{\R^d} \chi_R\,p(\rho_t)\,dx \\
  		&\hspace*{8.5em}+ \int_{\R^d} \nabla\chi_R\cdot (T_t^*-id)\,p(\rho_t)\,dx,
  	\end{aligned}
  \end{align}
where we used the fact that the distributional trace of the jacobian $\nabla\cdot T_t^*\ge 0$ is a nonnegative measure, and $\tilde\nabla \cdot $ represents the $dx$-absolutely continuous part of $\nabla\cdot $, defined in the Alexandrov almost everywhere sense, see \cite{mccann1997convexity} for details.
  
 Notice that by construction, the following convergences hold:
  \begin{align*}
	\left.\begin{aligned}
		\chi_R\langle T_t^*-id,\nabla p(\rho_t)\rangle &\to \langle T_t^*-id,\nabla p(\rho_t)\rangle \\
		\chi_R\,(\tilde\nabla \cdot T_t^*)\,p(\rho_t) &\to (\tilde\nabla\cdot T_t^*)\,p(\rho_t) \\
		\chi_R\,p(\rho_t) &\to p(\rho_t)
	\end{aligned}\;\;\right\}\;\; \text{as $R\to\infty$},
  \end{align*}
  for almost every $t\in(0,\infty)$. Furthermore, we know that $\tilde{\nabla}\cdot T_t^* \ge 0$, and that
  \begin{align*}
	\left.\begin{aligned}
		\|\chi_R\langle T_t^*-id,\nabla p(\rho_t)\rangle\|_{L^1(\R^d)} &\le \|\langle T_t^*-id,\nabla p(\rho_t)\rangle\|_{L^1(\R^d)} \\
		\|\chi_R\,p(\rho_t)\|_{L^1(\R^d)} &\le \|p(\rho_t)\|_{L^1(\R^d)}
	\end{aligned}\;\;\right\}\;\; \text{for all $R>0$}.
  \end{align*}
For the first bound, we used that $\nabla p(\rho_t)=\rho_t \nabla U'(\rho_t)$ to get
$$
\|\langle T_t^*-id,\nabla p(\rho_t)\rangle\|_{L^1(\R^d)} \leq W_2(\mu_t,\nu) \left(\int_{\R^d} |\nabla U'(\rho_t)|^2d\mu_t\right)^{1/2}\,.
$$
In particular, using the Lebesgue dominated convergence, we deduce
  \begin{align*}
  	\left.\begin{aligned}
	  	\int_{\R^d} \chi_R\langle T_t^*-id,\nabla p(\rho_t)\rangle\,dx &\to \int_{\R^d} \langle T_t^*-id,\nabla p(\rho_t)\rangle\,dx \\
	  	\int_{\R^d} \chi_R\,p(\rho_t)\,dx &\to \int_{\R^d} p(\rho_t)\,dx \\
	  	\int_{\R^d} \nabla\chi_R\cdot (T_t^*-id)\,p(\rho_t)\,dx &\to 0
  	\end{aligned}\;\;\right\}\;\; \text{as $R\to\infty$}.
  \end{align*}
As for the other term, we obtain from Fatou's lemma
  \[
   \int_{\R^d} (\tilde\nabla\cdot T_t^*)\,p(\rho_t)\,dx \le \liminf_{R\to\infty} \int_{\R^d} \chi_R\,(\tilde\nabla\cdot T_t^*)\,p(\rho_t)\,dx.
  \]
  Consequently, we pass to the limit as $R\to\infty$ in \eqref{eq:strong_int} to obtain the ``weak" integration by parts formula
  \begin{align*}
   -\int_{\R^d} \langle T_t^*-id,\nabla p(\rho_t)\rangle\,dx \ge \int_{\R^d} (\tilde\nabla \cdot T_t^* - d)\,p(\rho_t)\,dx,
  \end{align*}
  for almost every $t\in(0,\infty)$. Using the same arguments, we can perform a ``weak" integration by parts also for the other term in $I_1$, thereby obtaining
  \begin{align}\label{eq:weak_int}
  \begin{aligned}
  I_1 &\geq -\int_{\R^d} \big(d-(\tilde\nabla \cdot T_t)(y)\big)\,p(\omega)\,dy + \int_{\R^d} \big((\tilde\nabla \cdot T_t^*)(x)-d\big)\,p(\rho_t)\,dx \\
  &= \int_{\R^d} \Big((\tilde\nabla\cdot T_t)(y) + (\tilde\nabla\cdot T_t^*)(T_t(y))- 2d\Big)p(\omega)\,dy \ge 0,
  \end{aligned}
  \end{align}
where the last inequality follow from \cite{carrillo2006contractions,lisini2009nonlinear,bolley2012convergence}, where similar arguments were used. Finally, using the fact that $\nabla \mw(-x) = - \nabla \mw(x)$ for $x \in \R^d$, we can rewrite $I_2$ as 
  \begin{align*}
  I_2 &= \frac12\iint_{\R^d\times\R^d} \lal y - T_t(y) + T_t(\hat y) - \hat y, \nabla \mw(y-\hat y) - \nabla \mw(T_t(y) - T_t(\hat y)) \ral\, d\nu(y)d\nu(\hat y)\cr
  &= J_W(\mu_t|\nu).
  \end{align*}
  Putting all the terms together and invoking Theorem~\ref{thm:wasserstein_tt_2} concludes the proof.
\end{proof}

\begin{remark}
	In the isothermal case with $U(r)=r\log(r)$, $p'(r) = rU''(r) = 1$, and hence the assumption $p'(\rho_t)\in L^2(\mu_t)$ is trivially satisfied. As for the isentropic case, additional regularity is required. 
\end{remark}

A simple outcome of {\bf (H2)} is the following result which follows from direct computations using the convexity assumptions on the potentials, see \cite{carrillo2003kinetic} for more details.

\begin{proposition}\label{prop:J_dissipation}
 Under condition {\bf (H2)} for the potentials $V$ and $W$, we have that $J_V(\mu_t|\mu_\infty)$ and $J_W(\mu_t|\mu_\infty)$ defined in Corollary~\ref{cor:wasserstein_tt} are bounded from below, where $\mu_\infty\in\P_2^{ac}(\R^d)$ is a minimizer of the free energy $\F$ provided in Proposition~\ref{prop:stationary}. In particular,
 \begin{align*}
  J_V(\mu_t|\mu_\infty) &\ge c_V W_2^2(\mu_t,\mu_\infty),\\
  J_\mw(\mu_t|\mu_\infty) &\ge c_W W_2^2(\mu_t,\mu_\infty) - c_W\lt| \int_{\R^d} x\, d\mu_\infty - \int_{\R^d} x\, d\mu_t \rt|^2,
 \end{align*}
 where $c_V$ and $c_W$ are given in {\bf (H2)}.
\end{proposition}

We now have the essential ingredients to construct a Lyapunov functional for establishing the convergence to equilibrium of energy decaying solutions to the damped Euler equations \eqref{eq:euler} in the space $\P_2(\R^d)$ endowed with the 2-Wasserstein distance.

%
%

\section{Equilibration in Wasserstein distance}\label{sec:equilibration}

We begin this section by introducing the functionals involved and discuss their properties. The idea behind lies in the fact that the second temporal derivative of the Wasserstein distance produces a term on the right-hand side which gives a term that dissipates the Wasserstein distance itself (cf.~Proposition~\ref{prop:J_dissipation}). For this reason, we will have to include the term $dW_2^2/dt$ into the Lyapunov functional. In addition to the free energy $\F$, we consider the functionals
\begin{align*}
 \E(\mu_t,\mu_\infty) &:= W_2^2(\mu_t,\mu_\infty) + \int_{\R^d} |u_t|^2d\mu_t, \\
 \J(\mu_t,\mu_\infty)  &:= \alpha W_2^2(\mu_t,\mu_\infty) + \frac{d}{dt}W_2^2(\mu_t,\mu_\infty) + \beta\int_{\R^d} |u_t|^2d\mu_t,
\end{align*}
with constants $\alpha,\beta>0$. Since the $dW_2^2/dt$ term can be bounded from above by
\[
 \frac{1}{2}\left|\frac{d}{dt}W_2^2(\mu_t,\mu_\infty)\right| \le \int_{\R^d} |\langle y-T_t(y),u_t(T_t(y))\rangle|\,d\mu_\infty \le W_2(\mu_t,\mu_\infty)\lt(\int_{\R^d} |u_t|^2d\mu_t\rt)^{1/2},
\]
for all $t\ge 0$, where $T_t\#\mu_\infty=\mu_t$, we conclude that $\E$ and $\J$ are equivalent in the following sense:
\begin{align}\label{eq:equivalence}
 p\,\E \le \J \le q\,\E,
\end{align}
for constants $p,q>0$, depending only on $\alpha$ and $\beta$ whenever $\alpha\beta>1$.

\begin{remark}\label{rmk5}
 If $(\rho,u)$ is an energy decaying solution of the damped Euler equations \eqref{eq:euler}, then we obtain the uniform (in $t$) boundedness of $\|u_t\|_{L^2(\mu_t)}$ from the energy estimate \eqref{eq:energy_estimate}. Therefore, if $W_2(\mu_t,\mu_\infty)$ is also uniformly bounded in time, then 
 \[
  \sup\nolimits_{t\ge 0}\left|\frac{d}{dt}W_2^2(\mu_t,\mu_\infty)\right| \le M,
 \]
 for some constant $M<\infty$, which asserts that $t\mapsto W_2^2(\rho_t,\mu_\infty)$ is uniformly continuous.
\end{remark}

In order to provide the equilibration also for the velocity field $u$, we impose additional assumptions on the free energy $\F$:
\begin{enumerate}
 \item[{\bf (H3)}] The free energy $\F$ satisfies the stability estimate
 \begin{align}\label{eq:HWI}
  \F(\mu_t) - \F(\mu_\infty) \le c_\F(\mu_t) W_2(\mu_t,\mu_\infty),
 \end{align}
 for some time dependent function $c_\F(\mu_t)>0$ satisfying additionally
 \begin{align}\label{eq:regularity}
  \lim_{t\to \infty}\frac{1}{1+t}\int_0^t \frac{c_\F^2(\mu_s)\,}{1+s}\,ds = 0.
 \end{align}
\end{enumerate}

\begin{remark}
	A well-known inequality which takes the form \eqref{eq:HWI} is the so-called HWI inequality \cite{carrillo2003kinetic,carrillo2006contractions,villani2008optimal}:
	\[
	\F(\mu_t) - \F(\mu_\infty) \le \|\nabla(\delta_\mu\F)(\mu_t) \|_{L^2(\mu_t)}^2 W_2(\mu_t,\mu_\infty) - (\lambda/2)W_2^2(\mu_t,\mu_\infty),
	\]
	for some $\lambda \ge  0$. A classical example of a free energy satisfying the HWI inequality is given by
 \[
  \F(\rho) = \int_{\R^d} \rho\,\log \rho\,dx + \int_{\R^d} V(x)\rho\,dx,
 \]
where $V$ is a smooth convex potential such that $\int_{\R^d}\exp(-V(x))\,dx=1$. Then the corresponding stationary state is simply $\rho_\infty=e^{-V}$. Furthermore, since $\F(\rho_\infty)=0$, we have 
 \[
  0\le \F(\rho) - \F(\rho_\infty) = \int_{\R^d} \rho \,\log (\rho e^V)\,dx.
 \]
 Then from the standard HWI inequality \cite{villani2008optimal}, we obtain \eqref{eq:HWI} with $c_\F^2(\rho) = \int_{\mathbb{R}^d} |\nabla(\log \rho + V)|^2 \,d\rho$.
\end{remark}

\begin{remark}
\begin{enumerate}
 \item[(i)] Observe that {\bf (H3)} is also an assumption on the regularity of solutions to the damped Euler equations \eqref{eq:euler}. 
 
  \item[(ii)] A sufficient condition for \eqref{eq:regularity} includes the case $c_\F(\mu_t)\le c_\infty$ uniformly in time. Indeed,
  \[
   \frac{1}{1+t}\int_0^t \frac{c_\infty}{1+s}\,ds = c_\infty\,(1+t)^{-1}\ln(1+t) \longrightarrow 0\qquad\text{for\; $t\to\infty$}.
  \]
\end{enumerate}
\end{remark}

In view of condition {\bf (H2)}, we will study the equilibration for two separate cases.

\subsection{The case with confinement} 
Here, we consider the case where the confinement potential $V$ is present and satisfies condition {\bf (H2)}, as well as the interaction potential $W$. In this case, Proposition~\ref{prop:stationary} provides a stationary measure $\mu_\infty$, with density $\rho_\infty$, which satisfies 
\[
 \nabla p(\rho_\infty) + \rho_\infty\nabla\big( V + W\star\rho_\infty\big) = 0\qquad \text{a.e.}
\]
Hence, Theorem~1, with $\sigma=\mu_\infty$, and Proposition~\ref{prop:J_dissipation} holds true with
\[
 J_V(\mu_t|\mu_\infty) + J_W(\mu_t|\mu_\infty) \ge c_\ell W_2^2(\mu_t,\mu_\infty),\qquad c_\ell=\begin{cases}
  c_V & \text{for\; $c_W\ge 0$} \\
  c_V + c_W & \text{for\; $c_W< 0$}
 \end{cases}.
\]
Note that in the case $c_W\ge 0$, we have that $J_W(\mu_t|\mu_\infty)\ge 0$ due to Jensen's inequality.

\begin{theorem}\label{thm:confinement}
 Let $(\rho,u)$ be an energy decaying solution to the Euler equations \eqref{eq:euler}, with $\mu_t\in\P_2^{ac}(\R^d)$ the measure whose density is $\rho_t$ for all $t\geq0$, and $u$ satisfying additionally the assumptions of Theorem~\ref{thm:wasserstein_tt}, and $U$, $V \not\equiv 0$ and $W$ satisfying conditions {\bf (H1)}-{\bf (H2)}. Furthermore, assume that the initial data satisfies
 $
  \F(\mu_0) + W_2(\mu_0,\mu_\infty) <\infty\,,
 $
then 
 \[
  \lim_{t\to\infty} W_2(\mu_t,\mu_\infty) =0\,.
 \]
\end{theorem}
\begin{proof}
 Consider the functional
 \begin{align*}
  \G(\mu_t,\mu_\infty) &:=2\beta\big(\F(\mu_t)-\F(\mu_\infty)\big)+\J(\mu_t,\mu_\infty) \\
  &= \alpha W_2^2(\mu_t,\mu_\infty) + \frac{d}{dt}W_2^2(\mu_t,\mu_\infty) + 2\beta\big(\H(\rho_t,u_t)-\H(\rho_\infty,0)\big) \ge 0,
 \end{align*}
 with constants $\alpha,\beta>0$ to be specified later. Note that the first equality easily provides the nonnegativity of $\G$, while the second allows for the computation of the temporal derivative using \eqref{eq:energy_estimate}. Taking the temporal derivative of $\G$ along the flow generated by the damped Euler equations \eqref{eq:euler} and applying Corollary~\ref{cor:wasserstein_tt} we obtain
\begin{align*}
  \frac{d^+}{dt}\G(\mu_t,\mu_\infty) \le -2c_{\ell}W_2^2(\mu_t,\mu_\infty) + (\alpha-\gamma)\frac{d}{dt}W_2^2(\mu_t,\mu_\infty) - 2(\beta\gamma-1)\int_{\R^d} |u_t|^2d\mu_t .
 \end{align*}
Choosing $\alpha=\gamma$ and $\beta = (1+c_{\ell})/\gamma$, we further obtain
 \begin{align}\label{est_1}
  \frac{d^+}{dt}\G(\mu_t,\mu_\infty) \le -2c_{\ell}\E(\mu_t,\mu_\infty).
 \end{align}
Since $\alpha\beta=1+c_{\ell}>1$, we have the equivalence between $\J$ and $\E$, which concludes the proof. Indeed, integrating \eqref{est_1} over time interval $[0,t]$ gives
  \[
   \G(\mu_t,\mu_\infty) + 2c_\ell \int_0^t \E(\mu_s,\mu_\infty)\,ds \le \G(\mu_0,\mu_\infty),
  \]
Observe that $p W_2^2(\mu_t,\mu_\infty) \leq \G(\mu_t,\mu_\infty)\leq \G(\mu_0,\mu_\infty)$ due to \eqref{eq:equivalence}, which implies the uniform boundedness in time of $W_2^2(\mu_t,\mu_\infty)$, see Remark \ref{rmk5}.
  Since $\G$ is non-negative, we have that
  \[
   2c_\ell\int_0^\infty W_2^2(\mu_s,\mu_\infty)\,ds \le 2c_\ell \int_0^\infty \E(\mu_s,\mu_\infty)\,ds \le \G(\mu_0,\mu_\infty).
  \]
  Owing to the uniform continuity of $t\mapsto W_2^2(\mu_t,\mu_\infty)$ we obtain the asserted convergence \cite{kelman1960conditions}.
\end{proof}

Note that if $t\mapsto \|u_t\|_{L^2(\mu_t)}$ is further assumed to be uniformly continuous in $(0,\infty)$, then one also obtains $\|u_t\|_{L^2(\mu_t)}\to 0$ as $t\to\infty$. On the other hand, one may obtain the mentioned convergence under a different assumption provided in the following result.

\begin{corollary}\label{cor:confinement}
 Let $(\rho,u)$ be an energy decaying solution to the Euler equations \eqref{eq:euler}, with $\mu_t\in\P_2^{ac}(\R^d)$ the measure whose density is $\rho_t$ for all $t\geq0$, and $u$ satisfying additionally the assumptions of Theorem~\ref{thm:wasserstein_tt}, and $U$, $V \not\equiv 0$ and $W$ satisfying conditions {\bf (H1)}-{\bf (H3)}. Furthermore, assume that the initial data satisfies
 $
 \F(\mu_0) + W_2(\mu_0,\mu_\infty) <\infty\,,
 $
then 
 \[
\lim_{t\to\infty} \Bigl(W_2^2(\mu_t,\mu_\infty) + \|u_t\|_{L^2(\mu_t)}^2 \Bigr) =  \lim_{t\to\infty} \E(\mu_t,\mu_\infty) =0\,.
 \]
\end{corollary}
\begin{proof}
 Notice that for $t\ge 0$, we have
 \begin{align*}
 \frac{d^+}{dt}\big((1 + t)\,\G(\mu_t,\mu_\infty)\big) &= (1 + t)\frac{d^+}{dt}\G(\mu_t,\mu_\infty) + \G(\mu_t,\mu_\infty) \\
 &\le - 2(c_\ell/q)(1 + t)\, \J(\mu_t,\mu_\infty) + 2\beta\big(\F(\mu_t)-\F(\mu_\infty)\big) + \J(\mu_t,\mu_\infty) \\
 &\le - 2(c_\ell/q)(1 + t)\, \J(\mu_t,\mu_\infty) + 2\beta c_\F(\mu_t) W_2(\mu_t,\mu_\infty) + \J(\mu_t,\mu_\infty),
 \end{align*}
 where we used \eqref{est_1}, the equivalence \eqref{eq:equivalence} and {\bf (H3)}. Integrating the equation for $t\ge0$ gives
 \begin{align*}
 (1 + t)\,\G(\mu_t,\mu_\infty) + \frac{2c_\ell}{q}\int_0^t (1+s)\, \J(\mu_s,\mu_\infty)\,ds &\le c_0 + 2\beta \int_0^t c_\F(\mu_s) W_2(\mu_s,\mu_\infty)\,ds ,
 \end{align*}
 with the constant
 \[
 c_0:= \G(\mu_0,\mu_\infty) + \int_0^\infty \J(\mu_s,\mu_\infty)\,ds <\infty.
 \]
 For the second term on the right, we estimate from above using Young's inequality to obtain
 $$\begin{aligned}
 \int_0^t c_\F(\mu_t)W_2(\mu_s,\mu_\infty)\,ds &\le \frac{\e}{2}\int_0^t (1+s)W_2^2(\mu_s,\mu_\infty)\,ds + \frac{1}{2\e}\int_0^t \frac{c_\F^2(\mu_s)}{1+s}\,ds\cr
 & \le \frac{\e}{2p}\int_0^t (1+s)\J(\mu_s,\mu_\infty)\,ds + \frac{1}{2\e}\int_0^t \frac{c_\F^2(\mu_s)}{1+s}\,ds,
 \end{aligned}$$
 where we used the equivalence \eqref{eq:equivalence} again. Choosing $\e>0$ such that $\beta \e = p c_\ell /q$ yields
 \[
 (1 + t)\G(\mu_t,\mu_\infty) + \frac{c_\ell}{q}\int_0^t (1+s)\J(\mu_s,\mu_\infty)\,ds \le c_0 + c_1 \int_0^t \frac{c_\F^2(\mu_s)\,}{1+s}\,ds,
 \]
 which finally provides the required equilibration for $t\to\infty$.
\end{proof}

\subsection{The case with no confinement}
The case without confinement requires special attention since, in this case, the free energy $\F$ is translational invariant and consequently the center of mass is not a priori fixed. However, since the evolution of the center of mass \eqref{eq:oscillator} read
\begin{align*}
 \frac{d}{dt} \int_{\R^d} x\,d\mu_t = \int_{\R^d} u_t \, d\mu_t,\qquad \frac{d}{dt} \int_{\R^d} u_t \, d\mu_t = - \gamma \int_{\R^d} u_t\,d\mu_t,
\end{align*}
solving for the center of mass of $\mu_t$ clearly implies 
\begin{align*}
  \int_{\R^d} x\,d\mu_t = \int_{\R^d} x\,d\mu_0 + \frac{1}{\gamma}(1-e^{-\gamma t})\int_{\R^d} u_0\,d\mu_0 \longrightarrow \int_{\R^d} x\,d\mu_0 + \frac{1}{\gamma}\int_{\R^d} u_0\,d\mu_0\quad\text{as\; }t\to \infty.
\end{align*}
Owing to the limit above for the center of mass of $\mu_t$, we choose a minimizer $\mu_\infty$ of $\F$ satisfying
\[
 \int x\,d\mu_\infty = \int x\,d\mu_0 + \frac{1}{\gamma}\int u_0\,d\mu_0.
\]
For this choice of stationary measure $\mu_\infty\in\P_2^{ac}(\R^d)$ we have the following statement.

\begin{theorem}\label{thm:interaction}
Let $(\rho,u)$ be an energy decaying solution to the Euler equations \eqref{eq:euler}, with $\mu_t\in\P_2^{ac}(\R^d)$ the measure whose density is $\rho_t$ for all $t\geq0$, and $u$ satisfying additionally the assumptions of Theorem~\ref{thm:wasserstein_tt}, and $U$, $V\equiv 0$ and $W$ satisfying conditions {\bf (H1)}-{\bf (H2)}. Furthermore, assume that the initial data satisfies
 $
   \F(\mu_0) + W_2(\mu_0,\mu_\infty) <\infty\,,
$
then
  \[
 \lim_{t\to\infty}  W_2(\mu_t,\mu_\infty) = 0\,.
  \]
\end{theorem}
\begin{proof}
 As in Theorem \ref{thm:confinement}, we make use of Corollary~\ref{cor:wasserstein_tt}. In this particular case, we have $J_V\equiv 0$. As for $J_W$ we estimate as in Proposition~\ref{prop:J_dissipation} to obtain
 \[
  J_W(\mu_t|\mu_\infty) \ge c_W W_2^2(\mu_t,\mu_\infty) - c_W\lt| \int_{\R^d} x \,d\mu_\infty - \int_{\R^d} x \,d\mu_t \rt|^2.
 \]
 On the other hand, we have that
 \[
  \int_{\R^d} x \,d\mu_\infty - \int_{\R^d} x \,d\mu_t = \frac{1}{\gamma}e^{-\gamma t}\int_{\R^d} u_0\,d\mu_0,
 \]
 which subsequently gives
 \[
  J_W(\mu_t|\mu_\infty) \ge c_W W_2^2(\mu_t,\mu_\infty) - \frac{c_W}{\gamma^2}e^{-2\gamma t}\int_{\R^d} |u_0|^2 \,d\mu_0
 \]
 Taking the temporal derivative of the functional $\G$ provided in the proof of Theorem~\ref{thm:confinement} gives
\begin{align*}
   \frac{d^+}{dt}\G(\mu_t,\mu_\infty) \le -2c_W \E(\mu_t,\mu_\infty) + 2\frac{c_W}{\gamma^2}e^{-2\gamma t}\int_{\R^d} |u_0|^2 \,d\mu_0,
\end{align*}
where we chose $\alpha=\gamma$ and $\beta = (1+c_W)/\gamma$. Integrating the inequality above in time gives
\[
 \G(\mu_t,\mu_\infty) + 2c_W \int_0^t \E(\mu_s,\mu_\infty) \,ds \leq \G(\mu_0,\mu_\infty) + \frac{c_W}{\gamma^3}\int |u_0|^2 \,d\mu_0.
\]
Since $\G(\mu_t,\mu_\infty) \geq 0$ for all times $t\ge 0$, we finally obtain
\[
2c_W\int_0^\infty W_2^2(\mu_s, \mu_\infty)\,ds \le 2c_W\int_0^\infty \E(\mu_s, \mu_\infty)\,ds \leq \G(\mu_0,\mu_\infty) + \frac{c_W}{\gamma^3}\int |u_0|^2 \,d\mu_0 < \infty,
\]
and consequently the convergence due to the uniform continuity of $t\mapsto W_2^2(\mu_t,\mu_\infty)$.
\end{proof}

Proceeding as in the proof of Corollary~\ref{cor:confinement}, we obtain the following result.

\begin{corollary}\label{cor:interaction}
Let $(\rho,u)$ be an energy decaying solution to the Euler equations \eqref{eq:euler}, with $\mu_t\in\P_2^{ac}(\R^d)$ the measure whose density is $\rho_t$ for all $t\geq0$, and $u$ satisfying additionally the assumptions of Theorem~\ref{thm:wasserstein_tt}, and $U$, $V\equiv 0$ and $W$ satisfying conditions {\bf (H1)}-{\bf (H3)}. Furthermore, assume that the initial data satisfies
 $
	\F(\mu_0) + W_2(\mu_0,\mu_\infty) <\infty\,,
 $
then
	\[
\lim_{t\to\infty} \Bigl( W_2^2(\mu_t,\mu_\infty) + \|u_t\|^2_{L^2(\mu_t)} \Bigr) = \lim_{t\to\infty} \E(\mu_t,\mu_\infty) =0 \,.
	\]
\end{corollary}

%
%

\section{Overdamped limit $(\gamma\to\infty)$}\label{sec:relax}
In this section we consider the overdamped limit of \eqref{eq:euler} for large damping $\gamma\gg 1$. In this case, we rescale the time $t=\gamma\tilde t$, density $\rho_t=\tilde{\rho}_{\tilde t}$, and velocity $u_t=\tilde u_{\tilde t}/\gamma$. Dropping the tilde, we obtain the rescaled Euler equations
\begin{align}\label{eq:rescaled_euler}
\begin{aligned}
 \partial_t\rho_t^\gamma + \nabla\cdot(\rho_t^\gamma u_t^\gamma) &= 0,\\
 \partial_t(\rho_t^\gamma u_t^\gamma) + \nabla\cdot(\rho_t^\gamma u_t^\gamma\otimes u_t^\gamma) &= - \gamma^2\Big[\nabla p(\rho_t^\gamma) + \rho_t^\gamma\lt(\nabla V + \nabla \mw \star \rho_t^\gamma \rt) + \rho_t^\gamma u_t^\gamma\Big],
\end{aligned}
\end{align}
where we introduced the superscript $\gamma$ to make explicit the dependence of the solutions $(\rho^\gamma,u^\gamma)$ on the damping parameter $\gamma$. For the limit $\gamma\to\infty$, we wish to show that solutions $(\rho^\gamma,u^\gamma)$ corresponding to \eqref{eq:rescaled_euler} converge to the solution $(\bar\rho,\bar u)$ of the first order equation
\begin{align}\label{eq:porous_media}
 \partial_t\bar\rho_t + \nabla\cdot(\bar\rho_t \bar u_t) = 0,\qquad \bar u_t = -\nabla(\delta_\mu \F)(\bar\rho_t)\,.
\end{align}
It is well-known that the first order equation \eqref{eq:porous_media} has a gradient flow structure in $\P_2(\R^d)$ for the functional $\F$ and satisfies the decay estimates
\begin{align}\label{eq:porous_decay}
 \frac{d}{dt}\F(\bar\rho_t) = -\int_{\R^d} |\bar u_t|^2 \bar\rho_t \,dx,\qquad \frac{d}{dt}\int_{\R^d} |\bar u_t|^2\bar\rho_t\,dx = -D(\bar\rho_t),
\end{align}
with $D(\bar\rho_t)\ge 0$ for all times $t\ge 0$ \cite{carrillo2003kinetic}. Furthermore, $(\bar\rho,\bar u)$ satisfies the momentum equation
\[
 \partial_t(\bar\rho_t\bar u_t) + \nabla\cdot(\bar\rho_t\bar u_t\otimes\bar u_t) = -\bar\rho_t G_{\bar\rho_t},
\]
with $G_{\bar\rho_t} = -(\partial_t \bar u_t + \bar u_t\cdot\nabla \bar u_t)$ which we assume to be in $L^2((0,T),L^2(\bar\rho_t\, dx))$. In this case, given the measures $\mu^\gamma$ and $\bar{\mu}$, with densities $\rho^\gamma$ and $\bar{\rho}$ respectively, we get
\[
 \int_{\R^d} \langle y-T_t(y),G_{\bar\mu_t}(t,y)\rangle\,d\bar\mu_t \le \frac{1}{2}W_2^2(\mu_t^\gamma,\bar\mu_t) + \frac{1}{2}\int_{\R^d}|G_{\bar\mu_t}|^2 d\bar\mu_t,
\]
due to Young's inequality, where $T_t\#\bar\mu_t=\mu_t^\gamma$. On the other hand, we have
\begin{align*}
 \int_{\R^d} \langle y-T_t(y),G_{\mu_t^\gamma}(t,T_t(y))\rangle\,d\bar\mu_t \le -\gamma^2 \left[ c_\ell W_2^2(\mu_t^\gamma,\bar\mu_t) + \frac{1}{2}\frac{d}{dt}W_2^2(\mu_t^\gamma,\bar\mu_t) \right],
\end{align*}
where $c_\ell=c_V + \min\{c_W,0\}$. Here, we explicitly used the fact that $\bar u_t = -\nabla (\delta_\mu\F)(\bar\mu_t)$. Hence,
\begin{align}\label{eq:porous_tt}
\begin{aligned}
 \frac{d^+}{dt}\frac{d}{dt}W_2^2(\mu_t^\gamma,\bar\mu_t) \le &\, 2\int_{\R^d} |u_t^\gamma(T_t(y))-\bar u_t(y)|^2 d\bar\mu_t \\
 &- (2c_\ell \gamma^2 -1) W_2^2(\mu_t^\gamma,\bar\mu_t) - \gamma^2\frac{d}{dt}W_2^2(\mu_t^\gamma,\bar\mu_t) + \int_{\R^d}|G_{\bar\mu_t}|^2 d\bar\mu_t,
\end{aligned}
\end{align}
holds due to Theorem~\ref{thm:wasserstein_tt_2}. Now define the following functionals
\begin{align*}
 \E(\mu_t,\bar\mu_t) &:= W_2^2(\mu_t,\bar\mu_t) + \int_{\R^d}|u_t|^2 d\mu_t + \int_{\R^d} |\bar u_t|^2 d\bar\mu_t, \\
 \J(\mu_t,\bar\mu_t) &:= \alpha W_2^2(\mu_t,\bar\mu_t) + \frac{d}{dt}W_2^2(\mu_t,\bar\mu_t) + \beta\left[\int_{\R^d}|u_t|^2 d\mu_t + \int_{\R^d} |\bar u_t|^2 d\bar\mu_t\right],
\end{align*}
where $\alpha,\beta>0$ are appropriately chosen constants. Note that for $\alpha\beta>1/2$,
\[
 p\,\E \le \J \le q\,\E
\]
for some constants $p,q>0$, depending only on $\alpha$ and $\beta$.

\begin{theorem}\label{thm:relaxation}
Let $(\rho^\gamma,u^\gamma)$ be energy decaying solutions of \eqref{eq:rescaled_euler} for all $\gamma\gg 1$ sufficiently large and let $(\bar\rho,\bar u)$ be a gradient flow solution to \eqref{eq:porous_media} satisfying additionally the assumptions of Theorem~\ref{thm:wasserstein_tt_2}.  We further assume the initial conditions to be well-prepared, i.e., $\rho_0^\gamma=\bar\rho_0$ and $u_0^\gamma = \bar u_0$ such that
$
  \F(\bar\mu_0) + W_2(\mu_0^\gamma,\bar\mu_0) < \infty\,,
$
then
 \[
\lim_{\gamma\to\infty}  \int_0^T W_2^2(\mu_t^\gamma,\bar\mu_t)\,dt =0\,,
 \]
where $\mu^\gamma_t$ and $\bar{\mu}_t$ are the measures with densities $\rho^\gamma_t$ and $\bar{\rho}_t$ for all $t\geq 0$ respectively.
\end{theorem}
\begin{proof}
 The proof is based on Corollary~\ref{cor:wasserstein_tt} and follows the same line of arguments as in the previous theorem. We begin by defining the functional
 \[
  \G(\mu_t^\gamma,\bar\mu_t) := 2\beta \gamma^2 \big(\F(\mu_t^\gamma) + \F(\bar\mu_t)\big) + \J(\mu_t^\gamma,\bar\mu_t).
 \]
 The temporal derivative of $\G$ gives
 \begin{align*}
  \frac{d^+}{dt}\G(\mu_t^\gamma,\bar\mu_t) \le &\, (\alpha-\gamma^2)\frac{d}{dt}W_2^2(\mu_t^\gamma,\bar\mu_t) - 2\beta\gamma^2\left[ \int_{\R^d}|u_t^\gamma|^2 d\mu_t^\gamma + \int_{\R^d} |\bar u_t|^2 d\bar\mu_t \right] - \beta D(\bar\rho_t) \\
  & +2\int_{\R^d} |u_t^\gamma(T_t^\gamma(y))-\bar u_t(y)|^2 d\bar\mu_t - (2c_\ell \gamma^2-1)W_2^2(\mu_t^\gamma,\bar\mu_t) + \int_{\R^d}|G_{\bar\mu_t}|^2 d\bar\mu_t.
 \end{align*}
 where we used the estimates \eqref{eq:porous_decay} and \eqref{eq:porous_tt}. Choosing $\alpha=\gamma^2$ and noting that $D(\bar\rho_t)\ge 0$ and
 \[
  \int_{\R^d} |u_t^\gamma(T_t^\gamma(y))-\bar u_t(y)|^2 d\bar\mu_t \le 2\left[ \int_{\R^d} |u_t^\gamma|^2 d\mu_t^\gamma + \int_{\R^d} |\bar u_t|^2 d\bar\mu_t \right],
 \]
 we further obtain
 \begin{align*}
  \frac{d^+}{dt}\mathcal{G}(\mu_t^\gamma,\bar\mu_t) \le &\, - (2c_{\ell}\gamma^2-1) W_2^2(\mu_t^\gamma,\bar\mu_t) \\
  &- (\beta\gamma^2 - 2)\int_{\R^d} |u_t^\gamma(T_t^\gamma(y))-\bar u_t(y)|^2 d\bar\mu_t +\int_{\R^d}|G_{\bar\mu_t}|^2 d\bar\mu_t.
 \end{align*}
 We now choose $\beta= 2/\gamma^2$ and integrate in time $t$ over $[0,T]$ to obtain
 \begin{align*}
  \mathcal{G}(\mu_T^\gamma,\bar\mu_T) + (2c_{\ell}\gamma^2-1)\int_0^T W_2^2(\mu_t^\gamma,\bar\mu_t)\,dt
  \le \mathcal{G}(\mu_0^\gamma,\bar\mu_0) + \int_0^T\int_{\R^d}|G_{\bar\rho}|^2 d\bar\mu_t\,dt.
 \end{align*}
 Since $\F(\mu_t^\gamma) + \F(\bar\mu_t)\ge -c_0$ for all $t\ge 0$ for some constant $c_0\ge 0$,
 \begin{align}\label{eq:relaxation_estimate}
  \int_0^T W_2^2(\mu_t^\gamma,\bar\mu_t)\,dt  \le M(\gamma,T),
 \end{align}
 where
 \[
  M(\gamma,T) = \frac{1}{(2c_{\ell}\gamma^2-1)}\left(4c_0 + 8 \F(\bar\mu_0) + (4/\gamma^2)\int_{\R^d} |\bar u_0|^2 d\bar\mu_0 + \int_0^T\int_{\R^d}|G_{\bar\mu_t}|^2 d\bar\mu_t\,dt \right),
 \]
  whenever $\gamma^2>1/(2c_\ell)$. Passing to the limit $\gamma\to\infty$ concludes the proof.
\end{proof}

\begin{remark}
 Notice that for $\gamma^2>1/(2c_\ell)$, if $\sup_{0 \leq T < \infty}M(\gamma,T)<\infty$, the estimate \eqref{eq:relaxation_estimate} with $T\to\infty$ provides the convergence $W_2(\mu_t^\gamma,\bar{\mu}_t)\to 0$ for $t\to\infty$. Indeed, since $\int_0^\infty \|G_{\bar{\rho}}\|^2_{L^2(\bar{\mu}_t)}\,dt$ is finite, $M(\gamma)<\infty$. Hence, the claim follows again from the uniform continuity of $t\mapsto W_2^2(\mu_t^\gamma,\bar{\mu}_t)$. This means that the error between $\mu_t^\gamma$ and $\bar \mu_t$ in $W_2$ goes to zero as $t \to \infty$.
\end{remark}


\section{Rigorous examples in the 1D case}\label{sec:1d}

In spatial dimension one, we obtain an easy representation of the 2-Wasserstein distance given by the pseudo-inverse $\chi_t$ defined as follows. Let
\[
 F_t(x) = \int_{-\infty}^x d\mu_t = \mu_t((-\infty,x]) \in \Omega:=[0,1],
\]
be the cumulative distribution of the probability measure $\mu_t$. Then
\[
 \chi_t(\eta) = \inf\{x\in \R\,|\, F_t(x)>\eta\},
\]
defines the pseudo-inverse corresponding to $\mu_t$. In this case, the Wasserstein distance between two probability measures $\mu_t$ and $\nu_t$ is equivalently expressed as
\[
 W_2^2(\mu_t,\nu_t) = \int_\Omega |\chi_t(\eta)-\zeta_t(\eta)|^2d\eta,
\]
where $\chi_t$ and $\zeta_t$ are pseudo-inverses corresponding to $\mu_t$ and $\nu_t$ respectively.

The free energy corresponding to \eqref{eq:entropy} in terms of the pseudo-inverse $\chi_t$ is given by
\[
\F(\chi_t) = \frac{1}{m-1}\int_\Omega (\partial_\eta\chi_t(\eta))^{1-m} d\eta + \int_\Omega V(\chi_t(\eta))\,d\eta + \frac{1}{2}\int_{\Omega\times\Omega} W(\chi_t(\eta)-\chi_t(\bar\eta))\,d\bar\eta\, d\eta\,,
\]
and the entropy reads
\[
\H(\chi_t,v_t) = \F(\chi_t) + \frac{1}{2} \int_\Omega |v_t(\eta)|^2 d\eta.
\]

In this case, the damped isentropic Euler equations can be transformed into
\begin{subequations}\label{eq:1d_euler}
\begin{align}
 \partial_t \chi_t(\eta) &= u_t(\chi_t(\eta)) =: v_t(\eta),\\
 \partial_t v_t(\eta) &= -\partial_\eta\big( (\partial_\eta\chi_t(\eta))^{-m} \big) - (\partial_x V)(\chi_t(\eta)) - \int_\Omega (\partial_x W)(\chi_t(\eta)-\chi_t(\bar\eta))\,d\bar\eta - \gamma v_t(\eta),
\end{align}
\end{subequations}
for the Lagrangian quantities $(\chi_t,v_t)$ on $\Omega\times \R^+$ with initial condition $(\chi_0,v_0)$ on $\Omega$.

As before, a simple verification of the temporal derivative of $\H$ for smooth solutions gives
\[
 \frac{d}{dt}\H(\chi_t,v_t) = -\gamma\int_\Omega |v_t(\eta)|^2\,d\eta.
\]

\subsection{Temporal derivatives of the Wasserstein distance}
We begin by computing the second temporal derivative of the Wasserstein distance. The first temporal derivative reads
\begin{align*}
 \frac{1}{2}\frac{d}{dt}W_2^2(\mu_t,\nu_t) = \frac{1}{2}\frac{d}{dt}\int_\Omega |\chi_t(\eta)-\zeta_t(\eta)|^2 d\eta = \int_\Omega \langle \chi_t(\eta)-\zeta_t(\eta),v_t(\eta)-w_t(\eta)\rangle\,d\eta,
\end{align*}
where $v_t$ and $w_t$ are the velocities corresponding to $\chi_t$ and $\zeta_t$ respectively. This yields
\begin{align*}
 \frac{1}{2}\frac{d^2}{dt^2}W_2^2(\mu_t,\nu_t) &= \int_\Omega |v_t(\eta)-w_t(\eta)|^2d\eta + \int_\Omega \langle\chi_t(\eta)-\zeta_t(\eta),\partial_tv_t(\eta)-\partial_tw_t(\eta)\rangle\,d\eta \\
 &=: I_1 + I_2.
\end{align*}
The second term on the right-hand side gives
\begin{align*}
 I_2 = &\,\int_\Omega \langle \partial_\eta\chi_t(\eta)-\partial_\eta\zeta_t(\eta),(\partial_\eta\chi_t(\eta))^{-m}-(\partial_\eta\zeta_t(\eta))^{-m}\rangle\,d\eta - \frac{\gamma}{2} \frac{d}{dt} W_2^2(\mu_t,\nu_t)\\
 & - \int_\Omega \langle \chi_t(\eta)-\zeta_t(\eta),(\partial_x V)(\chi_t(\eta))-(\partial_x V)(\zeta_t(\eta))\rangle\,d\eta \\
 & - \int_{\Omega\times\Omega} \langle \chi_t(\eta)-\zeta_t(\eta),(\partial_x W)(\chi_t(\eta)-\chi_t(\bar\eta))-(\partial_x W)(\zeta_t(\eta)-\zeta_t(\bar\eta))\rangle\,d\bar\eta\,d\eta.
\end{align*}
Since the function $z^{-m}$, $m\ge 1$ is monotonically decreasing, we find that the first term is non-positive. Furthermore, due to the assumptions on $V$ and $W$ appeared in ({\bf H2}), we obtain
\begin{align*}
 I_2 \le -c_\ell \int_\Omega |\chi_t(\eta)-\zeta_t(\eta)|^2 d\eta,
\end{align*}
for some constant $c_\ell>0$ (cf. Proposition~\ref{prop:J_dissipation}). Hence, we finally obtain
\[
 \frac{1}{2}\frac{d^2}{dt^2}W_2^2(\mu_t,\nu_t) + \frac{\gamma}{2} \frac{d}{dt} W_2^2(\mu_t,\nu_t) \le \int_\Omega |v_t(\eta)-w_t(\eta)|^2d\eta - c_\ell W_2^2(\mu_t,\nu_t),
\]
which resembles the inequalities seen in Section~\ref{sec:equilibration}.

In the following examples, we only consider pressureless dynamics, i.e., the damped Euler system \eqref{eq:1d_euler} without internal pressure. In this case, the system for $(\chi_t,v_t)$ reads
\begin{subequations}\label{eq:1d_example}
\begin{align}
 \partial_t \chi_t(\eta) &= v_t(\eta), \\
 \partial_t v_t(\eta) &= - (\partial_x V)(\chi_t(\eta)) - \int_\Omega (\partial_x W)(\chi_t(\eta)-\chi_t(\bar\eta))\,d\bar\eta - \gamma v_t(\eta).
\end{align}
\end{subequations}

\subsection{Smooth solutions with repulsive Newtonian potential}
We consider the damped pressureless Euler for $(\chi_t,v_t)$ given by \eqref{eq:1d_example} with the explicit potentials $V(x) = |x|^2/2$ and $W(x)=-|x|$. Notice that $W$ now forms a repulsive interaction potential. In spatial dimension one, this interaction potential induces a free energy $\F$ on $\P_2(\R)$ that is known to be 1-convex along generalized geodesics \cite{carrillo2012mass}. The stationary solution is known to be $\chi_\infty=2\eta-1$ for the corresponding force $F[\chi_t]=\chi_t-\chi_\infty=:\zeta_t$. In this case, we consider the dynamics
\begin{align}\label{eq:smooth}
\partial_t \chi_t(\eta) = v_t(\eta),\qquad  \partial_t v_t(\eta) = -F[\chi_t](\eta) - \gamma v_t(\eta),
\end{align}
and only consider smooth solutions of the system, i.e., $\chi_t$ is strictly monotonically increasing up to sets of zero measure. Furthermore, since
\[
\frac{1}{2}\frac{d}{dt}\int_\Omega |\zeta_t|^2 d\eta = \int_\Omega \zeta_t v_t\,d\eta = \int_\Omega F[\chi_t]v_t\,d\eta  = \frac{d}{dt}(\F(\chi_t)-\F(\chi_\infty)),
\]
we deduce that $\F(\chi_t)-\F(\chi_\infty)$ is essentially $\|\zeta_t\|_{L^2(\Omega)}^2$ up to a constant shift, which allows us to only work with $\|\zeta_t\|_{L^2(\Omega)}^2$. The remaining derivatives may be easily computed to obtain
\begin{align*}
\frac{1}{2}\frac{d}{dt}\int_\Omega |v_t|^2 d\eta &= - \int_\Omega v_t\zeta_t\,d\eta - \gamma\int_\Omega |v_t|^2 d\eta, \\
\frac{d}{dt}\int_\Omega \zeta_t v_t\,d\eta &= \int_\Omega |v_t|^2d\eta - \gamma\int_\Omega \zeta_tv_t\,d\eta - \int_\Omega |\zeta_t|^2 d\eta.
\end{align*}
Now define the functional
\[
\G(\chi_t,\chi_\infty) := (\beta+\gamma)\int_\Omega |\zeta_t|^2 d\eta + 2\int_\Omega \zeta_t v_t\,d\eta + \beta\int_\Omega |v_t|^2 d\eta.
\]
Taking its temporal derivative gives
\[
\frac{d}{dt}\G = -2\int_\Omega |\zeta_t|^2 d\eta - 2(\beta\gamma - 1)\int_\Omega |v_t|^2 d\eta
\]
Choosing $\beta=2/\gamma$ and using the fact that $(\beta+\gamma)\beta = 2+\beta^2>1$, we have
\[
\frac{d}{dt}\G \le -(2/q)\G\quad\Longrightarrow\quad \G(\chi_t,\chi_\infty) \le e^{-(2/q)t}\G(\chi_0,\chi_\infty),
\]
for some positive constant $q=q(\gamma)$, i.e., $\G$ decays to zero at an exponential rate.
Summarizing the above discussions we have the following theorem.

\begin{theorem}Let $(\chi,v)$ be a global smooth solution of system \eqref{eq:1d_example} with potentials $V(x) = |x|^2/2$ and $W(x)=-|x|$. Suppose the initial data satisfies $\|\chi_0 - \chi_\infty\|_{L^2} +  \|v_0\|_{L^2} < \infty$ with $\chi_\infty=2\eta-1$, then
\[
\lim_{t \to \infty} \Bigl(\|\chi_t - \chi_\infty\|_{L^2} +  \|v_t\|_{L^2} \Bigr) = 0,
\]
exponentially fast. In particular, 
\[
 \lim_{t \to \infty} W_2(\mu_t,\mu_\infty) = 0,
\]
exponentially fast with $\mu_t=\chi_t\#\mathds{1}_{\Omega}d\eta$ and $\mu_\infty=\chi_\infty\#\mathds{1}_{\Omega}d\eta=(1/2)\mathds{1}_{(-1,1)}dx$.
\end{theorem}

\begin{remark}
	The system \eqref{eq:smooth} has been extensively studied in \cite{carrillo2016pressureless} via the characteristics formulation with a friction coefficient of $\gamma=1$ and is known to either have smooth solutions or solutions may blow up in finite time, depending on the total mass. 
\end{remark}

\subsection{Generalized Lagrangian solutions with attractive Newtonian potential}
The notion of sticky solutions of pressureless Euler systems have been considered since the $70$'s to describe $\delta$-shocks that may form in finite time \cite{zel1970gravitational}. Since then, numerous works have gone into the construction of such solutions, thereby extending the notion of a solution of \eqref{eq:1d_example} past the formation of $\delta$-shocks \cite{brenier2013sticky,brenier1998sticky,carnevale1990statistics,chertock2007new,weinan1996generalized,natile2009wasserstein}. Here, we adopt the notion of generalized Lagrangian solutions for globally sticky dynamics found in \cite{brenier2013sticky}.

Consider the potentials $V(x)=|x|^2/2$ and $W(x)=|x|$, which provides global attraction for the Euler system \eqref{eq:1d_example} with the free energy
\[
 \F(\chi_t) = \frac{1}{2}\int_\Omega |\chi_t|^2d\eta + \frac{1}{2}\int_{\Omega\times\Omega} |\chi_t(\eta)-\chi_t(\bar\eta)|\,d\bar\eta\, d\eta. 
\]
Due to sufficiently strong attraction, one expects the formation of $\delta$-shocks in finite time and the only stable stationary solution of \eqref{eq:1d_example} is the Dirac measure $\mu_\infty=\delta_0$ at $x=0$. This corresponds to the stable stationary pseudoinverse $\chi_\infty\equiv 0$. 

To describe the sticky dynamics, we denote by $\K$ the closed convex cone of right-continuous nondecreasing functions in $L^2(\Omega)$, i.e.,
\[
 \K = \left\{\chi\in L^2(\Omega) \;\big|\; \chi\;\;\text{is nondecreasing}\right\},
\]
and $I_\K\colon L^2(\Omega)\to [0,+\infty]$ be the indicator function of $\K$ which is convex and lower semicontinuous. Hence, its subdifferential $\partial I_\K(\chi)$ at $\chi\in\K$ is a maximal monotone operator on $L^2(\Omega)$ and it can be characterized as
\[
 \partial I_\K(\chi) = \left\{ \zeta\in L^2(\Omega)\;\Big|\; \int_\Omega \zeta(\bar\chi-\chi)\,d\eta \le 0 \;\;\text{for all\; $\bar\chi\in \K$}\right\}.
\]
Now define the set
\[
 \Omega_\chi := \left\{ \eta\in \Omega\;|\; \chi\;\; \text{is constant in an open neighborhood of $\eta$} \right\},
\]
and the closed subspace
\[
 H_\chi = \left\{ \zeta\in L^2(\Omega)\;|\; \zeta\;\;\text{is constant on each interval $(a,b)\subset\Omega_\chi$} \right\},
\]
for any $\chi\in \K$. The projection $\mathsf{P}_{\chi}\colon L^2(\Omega)\to H_\chi$ is given by $\mathsf{P}_{\chi}(\zeta)=\zeta$ a.e.~in $\Omega\setminus\Omega_\chi$ and
\[
 \mathsf{P}_{\chi}(\zeta) = \frac{1}{b-a}\int_a^b \zeta(\eta)\,d\eta\quad\text{in any maximal interval $(a,b)\subset\Omega_\chi$}.
\]
It was shown in \cite{brenier2013sticky} that the tangent cone $\mathcal{T}_\chi\K$ to $\K$ at $\chi\in\K$ can be characterized as
\[
 \mathcal{T}_\chi\K = \left\{ \zeta\in L^2(\Omega)\;|\; \zeta\;\;\text{is nondecreasing in each interval $(a,b)\subset\Omega_\chi$} \right\}.
\]
In particular, $\mathsf{P}_\chi(\zeta)\in \mathcal{T}_\chi\K$ for any $\zeta\in L^2(\Omega)$.

Therefore, a global sticky dynamics for \eqref{eq:1d_example} can be formulated as 
\begin{align}\label{eq:sticky}
 \partial_t \chi_t(\eta) = \textsf{P}_{t}(v_t)(\eta),\qquad  \partial_t v_t(\eta) = -\textsf{P}_t(F[\chi_t])(\eta) - \gamma v_t(\eta),
\end{align}
where $F[\chi_t](\eta) = \chi_t(\eta) + 2\eta -1$ are the external and interaction forces acting on the system, and $\mathsf{P}_t := \mathsf{P}_{\chi_t}$ is the projection onto the closed subspace $H_{\chi_t}$ defined above. For this particular choice of potentials, there is a unique sticky Lagrangian solution $(\chi_t,v_t)$ of \eqref{eq:sticky} with
\[
 \chi \in \text{Lip}_{\text{loc}}(\R^+,\K),\qquad v\in \C^1(\R^+,L^2(\Omega)),
\]
and initial data $(\chi_0,v_0)\in \K\times L^2(\Omega)$.

We now compute the evolution of the energy corresponding to the solution pair $(\chi_t,v_t)$:
\[
 \frac{d^+}{dt}\left[\F(\chi_t) + \frac{1}{2}\int_\Omega |v_t|^2\,d\eta\right] = \int_\Omega F[\chi_t]\textsf{P}_t(v_t)\,d\eta - \int_\Omega v_t \mathsf{P}_t(F[\chi_t])\,d\eta - \gamma\int_\Omega |v_t|^2\,d\eta.
\]
Notice that the following equalities hold
\[
 \int_\Omega \mathsf{P}_t(v_t)(\mathsf{P}_t(F[\chi_t])-F[\chi_t]))\,d\eta = 0 = \int_\Omega \mathsf{P}_t(F[\chi_t])(\mathsf{P}_t(v_t)-v_t)\,d\eta.
\]
This consequently yields
\[
 \int_\Omega F[\chi_t]\textsf{P}_t(v_t)- v_t \mathsf{P}_t(F[\chi_t])\,d\eta = \int_\Omega\mathsf{P}_t(v_t)\mathsf{P}_t(F[\chi_t]) - \mathsf{P}_t(F[\chi_t])\mathsf{P}_t(v_t)\,d\eta = 0,
\]
and the above equality can be simplified to 
\[
 \frac{d^+}{dt}\left[\F(\chi_t) + \frac{1}{2}\int_\Omega |v_t|^2\,d\eta\right] = - \gamma\int_\Omega |v_t|^2\,d\eta,
\]
which gives the same expression as the usual case and also implies the uniform temporal bounds on $\F(\chi_t)$ and $\|v_t\|_{L^2(\Omega)}$. In particular, we have the uniform estimate
\[
 \sup\nolimits_{t\ge 0} \Big(\F(\chi_t) + \frac12\|v_t\|_{L^2(\Omega)}^2\Big) \le \F(\chi_0) + \frac12\|v_0\|_{L^2(\Omega)}^2=: \frac{c_0}{2}.
\]
Consequently, we conclude that
\[
 \F(\chi_t) \le \frac{1}{2}\int_\Omega |\chi_t|^2d\eta + \int_\Omega |\chi_t|\,d\eta \le c_\F\|\chi_t\|_{L^2(\Omega)}
\]
with $c_\F=1+(\sqrt{c_0}/2)$ and this satisfies the required assumption in {\bf (H3)}.

Now taking the temporal derivative of the $L^2$-distance between $\chi_t$ and $\chi_\infty=0$ gives
\[
 \frac{1}{2}\frac{d^+}{dt}\int_\Omega |\chi_t|^2 d\eta = \int_\Omega \chi_t\mathsf{P}_t(v_t)\,d\eta = \int_\Omega \chi_t(\mathsf{P}_t(v_t)-v_t)\,d\eta + \int_\Omega \chi_t v_t\,d\eta = \int_\Omega \chi_t v_t\,d\eta,
\]
where we used the fact that $\chi_t\in H_{\chi_t}$ and $\mathsf{P}_t(v_t)-v_t\in H_{\chi_t}^\bot$. Since $\chi_t$ is only locally Lipschitz continuous in time, we are unable to consider the second temporal derivative. Instead we compute 
\begin{align*}
 \frac{d^+}{dt}\int_\Omega \chi_t v_t\,d\eta &= \int_\Omega \mathsf{P}_t(v_t)v_t\,d\eta - \gamma\int_\Omega \chi_t v_t\,d\eta - \int_\Omega \chi_t\mathsf{P}_t(F[\chi_t])\,d\eta \\
 &\le \int_\Omega |v_t|^2d\eta - \gamma\int_\Omega \chi_t v_t\,d\eta - \int_\Omega \chi_t F[\chi_t]\,d\eta,
\end{align*}
where we used the nonexpansive property of the projection $\mathsf{P}_t$ for all times $t\ge 0$ in the first term. Notice that the last term can be rewritten as
\[
 \int_\Omega \chi_t F[\chi_t]\,d\eta = \int_\Omega |\chi_t|^2\,d\eta + \int_\Omega \chi_t(2\eta-1)\,d\eta = \int_\Omega \chi_t^2\,d\eta - \int_\Omega \partial_\eta\chi_t(\eta^2-\eta)\,d\eta.
\]
Since $\chi_t\in \K$ and $\eta(\eta-1)\le 0$ for all $\eta\in\Omega$, we have that
\[
  \int_\Omega \partial_\eta\chi_t(\eta^2-\eta)\,d\eta \le 0,
\]
and therefore
\[
 \frac{d^+}{dt}\int_\Omega \chi_t v_t\,d\eta \le  \int_\Omega |v_t|^2d\eta - \gamma\int_\Omega \chi_t v_t\,d\eta - \int_\Omega |\chi_t|^2\,d\eta.
\]
From here, we may proceed as in the previous section to conclude convergence to equilibrium for the global sticky dynamics towards the unique stationary solution $\delta_0\in\P_2(\R)$. We can summarize the discussion above in the following result.

\begin{theorem} Let $(\chi,v)$ be a global Lagrangian solution of system \eqref{eq:1d_example} with potentials $V(x) = |x|^2/2$ and $W(x)=|x|$. Suppose the initial data satisfies $\|\chi_0\|_{L^2} +  \|v_0\|_{L^2} < \infty$, then
\[
\lim_{t \to \infty} \Bigl(\|\chi_t\|_{L^2} +  \|v_t\|_{L^2} \Bigr) = 0\,.
\]
In particular, 
\[
\lim_{t \to \infty} W_2(\mu_t,\mu_\infty) = 0,
\]
with $\mu_t=\chi_t\#\mathds{1}_{\Omega}d\eta$ and $\mu_\infty=\delta_0$.
\end{theorem}

\section*{acknowledgements}
JAC was partially supported by the EPSRC grant number EP/P031587/1. YPC was supported by National Research Foundation of Korea(NRF) grant funded by the Korea government(MSIP) (No. 2017R1C1B2012918 and 2017R1A4A1014735) and POSCO Science Fellowship of POSCO TJ Park Foundation. The authors are very grateful to the Mittag-Leffler Institute for providing a fruitful working environment during the special semester \emph{Interactions between Partial Differential Equations \& Functional Inequalities}.

\bibliographystyle{abbrv}
\bibliography{wasserstein}

\end{document}